\documentclass[12pt,reqno]{amsart}
\usepackage{amssymb, a4wide, amsmath, mathtools,xy}
\xyoption{all}
\usepackage{latexsym,pdfsync,xcolor,graphicx}
\usepackage{multirow}

\usepackage{array}
\usepackage{upgreek}
\usepackage{pstricks-add}
\usepackage{enumerate}

\usepackage[english]{babel}	
\usepackage{comment}
\allowdisplaybreaks

\usepackage{pgf,tikz}
\usepackage{wrapfig}
\usetikzlibrary{arrows}
\usepackage{orcidlink}

\usepackage{float}
\usepackage{appendix}

\usepackage{hyperref}
\usepackage{rotating}

\makeatletter
\def\l@section{\@tocline{1}{10pt}{0em}{2.5em}{\bfseries}}
\def\l@subsection{\@tocline{2}{0pt}{2.5em}{3.5em}{}}
\def\l@subsubsection{\@tocline{3}{0pt}{5em}{5.5em}{}}
\makeatother

\newtheorem{theorem}{Theorem}[section]

\newtheorem{corollary}[theorem]{Corollary}
\newtheorem{proposition}[theorem]{Proposition}
\newtheorem{notation}[theorem]{Notation}
\newtheorem{lemma}[theorem]{Lemma}

\theoremstyle{definition}
\newtheorem{definition}[theorem]{Definition}
\newtheorem{example}[theorem]{Example}

\theoremstyle{remark}
\newtheorem{remark}[theorem]{Remark}

\DeclareMathOperator{\ann}{ann}
\DeclareMathOperator{\spa}{span}

\DeclareMathOperator{\rank}{rank}
\DeclareMathOperator{\supp}{supp}

\DeclareMathOperator{\codim}{codim}
\DeclareMathOperator{\asoc}{Asoc}
\DeclareMathOperator{\nil}{Nil}
\DeclareMathOperator{\bnil}{BNil}
\DeclareMathOperator{\snil}{SNil}

\numberwithin{equation}{section}
\title[A Frattini theory for evolution algebras]{A Frattini theory for evolution algebras}

\author[M. Ladra]{Manuel Ladra\textsuperscript{1}\,\orcidlink{0000-0002-0543-4508}}
\address{\textsuperscript{1}Department of Mathematics \& CITMAga, Universidade de Santiago de Compostela, 15782 Santiago de Compostela, Spain}
\email{manuel.ladra@usc.es}

\author[A. Pérez-Rodríguez]{Andrés Pérez-Rodríguez\textsuperscript{2}\,\orcidlink{0009-0007-1095-5328}}
\address{\textsuperscript{2}Department of Mathematics \& CITMAga, Universidade de Santiago de Compostela, 15782 Santiago de Compostela, Spain}
\email{andresperez.rodriguez@usc.es}



\begin{document}

\begin{abstract}
	This paper develops a Frattini theory for evolution algebras defining the Frattini subalgebra as the intersection of all maximal subalgebras, and the Frattini ideal as the largest ideal contained in it. To this end, we revisit the notion of nilradical, whose classical definition is not directly applicable in this setting, and propose the supersolvable nilradical as a suitable alternative. This leads to necessary and sufficient conditions for the triviality of the Frattini subalgebra and ideal. Finally, we also briefly examine the relevance of the Frattini ideal in the study of dually atomistic evolution algebras.
\end{abstract}

\maketitle

\tableofcontents

\section{Introduction}

\textit{Frattini theory} originates in group theory, where the \textit{Frattini subgroup}, defined as the intersection of all maximal subgroups of a given group, was first introduced by Giovanni Frattini in 1885 (see \cite{F_1885_origin}). The Frattini subgroup captures the notion of the group's ``non-generators'', as it consists precisely of those elements that can be removed from any generating set without losing the generation of the group. In addition, this subgroup has several other remarkable properties; particularly it is always a characteristic subgroup and, consequently, it is normal. Furthermore, the Frattini subgroup of a finite group is nilpotent. Last but not least, it also serves as a criterion for nilpotency, as it contains the derived subgroup if and only if the group itself is nilpotent (see, for instance, \cite[pp.~156--158]{H_59_grouptheory}).

Over time, Frattini theory has been extended beyond group theory to establish a parallel framework for algebras. The similarities between Lie algebras and groups led to the analogous definition of the Frattini subalgebra as the intersection of all maximal subalgebras. This extension motivated numerous investigations, notably by Barnes (see~\cite{B_67_cohom}), Barnes and Gastineau-Hills (see~\cite{BG_68_Lie}), and Marshall (see~\cite{M_67_frat}), among many others. Although less explored, related studies have also emerged in the broader context of non-associative algebras (see~\cite{frat_towers}), and more specifically, in the areas of Leibniz algebras (see~\cite{frat_leib}) and restricted Lie algebras (see~\cite{LT_85_rest}). However, to the best of our knowledge, no similar study has been conducted in the setting of genetic algebras and, in particular, evolution algebras.

Evolution algebras are commutative but non-associative structures which are not defined by identities. They were introduced by J.~P.~Tian and P.~Vojt\v{e}chovsk\'{y} in $2006$ (see \cite{TV_06}) to model non-Mendelian inheritance, which is actually considered the basic language of molecular biology.
Two years later, J.~P.~Tian expanded on this work in a monograph (see \cite{Tian_08}), providing a comprehensive study of their algebraic properties and biological applications. 
This work laid the foundation for numerous subsequent studies on the structural properties of evolution algebras. In particular, researchers have investigated their ideals (see \cite{BCS_22_natural, CKS_19_basic}), conditions to be simple and semisimple (see \cite{CMMT_25}), and even their connections with graph theory (see \cite{EL_15}), as well as several other structural aspects.

Motivated by these developments and many subsequent ones, this work aims to establish a Frattini theory for evolution algebras, also defining their Frattini subalgebras as the intersection of all their maximal subalgebras, and its Frattini ideal as the largest ideal contained within the Frattini subalgebra. To approach this, we begin by revisiting the concept of the \textit{nilradical}, which, in general, is not well-defined in this context. Consequently, we introduce the notion of supersolvable nilradical, which will allow us to establish necessary and sufficient conditions for the Frattini subalgebra and ideal to be trivial. Finally, we briefly explore the role of the Frattini subalgebra in the study of dually atomistic evolution algebras.

The paper is organised into five sections. Following this introduction, Section~\ref{sec:2} presents the preliminary material. We first recall fundamental concepts of evolution algebras, focusing particularly on solvable evolution algebras with one-dimensional derived subalgebras. The section concludes with a review of the essential background on Frattini theory that is necessary for the following developments.

We begin Section~\ref{sec:3} illustrating that, in the context of evolution algebras, multiple maximal nilpotent ideals may exist (see Example~\ref{ex:1}). Consequently, the nilradical of an evolution algebra cannot be defined as its largest nilpotent ideal, unlike in other non-associative structures such as Lie or Leibniz algebras. The main goal of this section is to provide a suitable definition of the nilradical in this setting. To this end, we first focus on solvable evolution algebras with one-dimensional derived subalgebras, where the maximal nilpotent ideal is unique (see Theorem~\ref{th:nilrad_dim1}). Using this result, we inductively define a series of nilpotent ideals leading to the notion of the supersolvable nilradical (Definition~\ref{def:super_nilp}), which captures the essence of the nilradical while enjoying several desirable properties (see Proposition~\ref{prop:sumsnil} and Corollary~\ref{cor:snil_max}).

Subsequently, Section \ref{sec:4} is dedicated to the study of Frattini theory. We characterise both the Frattini subalgebra and the Frattini ideal in the context of solvable evolution algebras with one-dimensional derived subalgebras (see Theorem~\ref{th:nilrad_abelian}). We then establish a necessary condition for an evolution algebra to be $\phi$-free in terms of its supersolvable nilradical (see Theorem~\ref{th:phi-free1}). Finally, we also prove the sufficiency of these conditions in a particular case related to the support of the supersolvable nilradical (see Theorem~\ref{th:equiv_frat}). In the end,
Section~\ref{sec:5} concludes this work, where we characterise the property of being dually atomistic within specific families of evolution algebras (Theorem \ref{th:aa_dually}) by introducing the concept of almost (basic) abelian evolution algebras.

\section{Preliminaries}\label{sec:2}
We begin by establishing the basic notation used throughout this paper. $\mathbb{K}$ will denote an arbitrary field of characteristic different from two, and $\mathbb{K}^*$ will stand for $\mathbb{K}\backslash\{0\}$. Given a $\mathbb{K}$-algebra $\mathcal{A}$ and a subset $S\subseteq\mathcal{A}$, we denote by $\spa\{S\}$ the $\mathbb{K}$-linear span of $S$. Additionally, we use $+$ and $\oplus$ to denote sums and direct sums of vector spaces, respectively.

\subsection{Preliminaries on evolution algebras}
An \textit{evolution algebra} over $\mathbb{K}$ is a $\mathbb{K}$-algebra $\mathcal{E}$ that admits a basis $B=\{e_1,\dots,e_n,\dots\}$, called \textit{natural basis}, such that $e_ie_j=0$ for all $i\neq j$. In this note, we focus on finite-dimensional evolution algebras, meaning that $B$ is a finite set. For a given natural basis $B=\{e_1,\dots,e_n\}$ in $\mathcal{E}$, the scalars $a_{ij}\in\mathbb{K}$ satisfying $e_i^2=\sum_{j=1}^na_{ij}e_j$ are called the \textit{structure constants} of $\mathcal{E}$ relative to $B$. The matrix $M_B(\mathcal{E})=(a_{ij})_{i,j=1}^n$ is said to be the \textit{structure matrix} of $\mathcal{E}$ relative to $B$. Moreover, recall that the \textit{annihilator} of an evolution algebra $\mathcal{E}$ is characterised by \cite[Proposition~1.5.3]{tesisyolanda},
\[\ann(\mathcal{E})\coloneqq\{u\in\mathcal{E}\colon u\mathcal{E}=0\}=\spa\{e_i\in B\colon e_i^2=0\}.\]

Given an evolution algebra $\mathcal{E}$ with a natural basis $B=\{e_1,\dots,e_n\}$ and an element $u=\sum_{i=1}^n\mu_ie_i\in\mathcal{E}$, we define its \textit{support} relative to $B$ as $\supp_B(u)\coloneqq\{i:\mu_i\neq0\}$. In a similar way, we define the \textit{support of a subspace} $U\subset\mathcal{E}$ as $\supp_B(U)\coloneqq\cup_{u\in U}\supp(u)$. Clearly, if $B_U$ is a basis of $U$, then $\supp_B(U)=\cup_{u\in B_U}\supp(u)$. When the choice of the natural basis is clear, we simply write $\supp$ instead of $\supp_B$.

\begin{remark}\label{rem:id}
	Let $\mathcal{E}$ be an evolution algebra with a natural basis $B=\{e_1,\dots,e_n\}$. If $I$ is an ideal of $\mathcal{E}$, then $\spa\{e_i\colon i\in\supp(I)\}$ is also an ideal. 
\end{remark}

An ideal $I$ of $\mathcal{E}$ is called a \textit{basic ideal} of $\mathcal{E}$ relative to $B$ if it admits a natural basis consisting of vectors from $B$. Moreover, adopting the standard terminology from algebras, an ideal $I$ of $\mathcal{E}$ is said to be an \textit{abelian ideal} if $I^2=0$; and an evolution algebra $\mathcal{E}$ is said to be \textit{abelian} if $\mathcal{E}^2=0$. For instance, $\ann(\mathcal{E})$ is clearly an abelian ideal. Moreover, an evolution algebra is said to be \textit{almost abelian} if it is nonabelian, but it has an abelian ideal of codimension one.

An evolution algebra $\mathcal{E}$ is said to be \textit{simple} if it is nonabelian and it does not have proper ideals; and it is called \textit{semisimple} if it is a direct sum of simple evolution algebras. Following~\cite{CSV_16_semiprime}, an evolution algebra $\mathcal{E}$ is  called \textit{semiprime} if it has no nonzero abelian ideals.
We also say that an evolution algebra $\mathcal{E}$ is \textit{supersolvable} if there exists a complete flag made up of ideals, that is, there exists a chain of ideals \[0=I_0\subsetneq I_1\subsetneq\dots\subsetneq I_n=\mathcal{E}\] such that $\dim{I_i}=i$ for every $0\leq i\leq n$.

\begin{remark}
If an evolution algebra $\mathcal{E}$ is supersolvable, then every maximal subalgebra of $\mathcal{E}$ has codimension one. This result for Lie algebras can be found in \cite[Theorem~7]{B_67_cohom}, but the proof works in general. Let $M$ be a maximal subalgebra of $\mathcal{E}$ and $I$ a one-dimensional ideal of $\mathcal{E}$. We proceed by induction on $\dim{\mathcal{E}}$. If $I \subseteq M$, then $M/I$ is a maximal subalgebra of $\mathcal{E}/I$ with codimension one; hence $\codim{M}=1$. If instead, $I\nsubseteq M$, then $\mathcal{E}=M+I$, $M\cap I=0$, which implies again that $\codim{M}=\dim{I} = 1$.
\end{remark}

Next, we also introduce the concept of an $\mathcal{E}$-supersolvable ideal, which is actually inspired by the notion of a $G$-supersolvable normal subgroup.

\begin{definition}
	Let $\mathcal{E}$ be an evolution algebra, and let $I$ be an ideal of $\mathcal{E}$. We say that $I$ is an \textit{$\mathcal{E}$-supersolvable ideal} if it admits a complete flag made up of ideals of $\mathcal{E}$. That is, there exists a chain \[0=I_0\subsetneq I_1\subsetneq\dots\subsetneq I_r=I\] such that each $I_i$ is an ideal of $\mathcal{E}$ and $\dim{I_i}=i$ for every $0\leq i\leq r$. 
\end{definition}
\begin{remark}
 The notion of $\mathcal{E}$-supersolvability enjoys several useful properties.
 Let $\mathcal{E}$ be an evolution algebra and let $I,J$ be ideals of $\mathcal{E}$ such that $J\subset I$. Then, the following assertions follow easily from the definition of $\mathcal{E}$-supersolvability:
\begin{itemize}
	\item If $\mathcal{E}$ is supersolvable, then every ideal of $\mathcal{E}$ is clearly $\mathcal{E}$-supersolvable.
	\item If $I$ is $\mathcal{E}$-supersolvable, then  $J$ is $\mathcal{E}$-supersolvable, and $I/J$ is $\mathcal{E}/J$-supersolvable.
	\item If $J$ is $\mathcal{E}$-supersolvable and $I/J$ is $\mathcal{E}/J$-supersolvable, then $I$ is also $\mathcal{E}$-supersolvable.
	\item The sum of $\mathcal{E}$-supersolvable ideals is again $\mathcal{E}$-supersolvable.
\end{itemize}
\end{remark}

In this paper we will work extensively with nilpotent and solvable evolution algebras. Given a (not necessarily evolution) algebra $\mathcal{E}$, we define the following sequences of subalgebras:
\begin{align*}
		&\mathcal{E}^{\langle1\rangle}=\mathcal{E},\qquad \quad  \mathcal{E}^{\langle k+1\rangle}=\mathcal{E}^{\langle k\rangle}\mathcal{E};	\\	
		&\mathcal{E}^1=\mathcal{E},\qquad  \quad  \ \ \mathcal{E}^{k+1}=\sum_{i=1}^k\mathcal{E}^i\mathcal{E}^{k+1-i};	\\
		&\mathcal{E}^{(1)}=\mathcal{E},\qquad \quad  \mathcal{E}^{(k+1)}=\mathcal{\mathcal{E}}^{(k)}\mathcal{E}^{(k)}.
	\end{align*}
An (evolution) algebra $\mathcal{E}$ is called \textit{right nilpotent} if there exists $n\in\mathbb{N}$ such that
 $\mathcal{E}^{\langle n\rangle}=0$, \textit{nilpotent} if there exists $n\in\mathbb{N}$ such that $\mathcal{E}^{n}=0$,
  and \textit{solvable} if there exists $n\in\mathbb{N}$ such that $\mathcal{E}^{(n)}=0$. The second term in each of the
  previous sequences, $\mathcal{E}^{\langle2\rangle}=\mathcal{E}^{2}=\mathcal{E}^{(2)}$, is called the \textit{derived
  	 subalgebra of $\mathcal{E}$}. Recall that a commutative algebra is right nilpotent if and only if it is nilpotent (see \cite[Chapter 4, Proposition 1]{ZSSS_82}). Moreover, the structure matrix of a nilpotent evolution algebra can be assumed to be strictly (upper or lower) triangular by \cite[Theorem~2.7]{CLOR_14}. 
We now introduce the following notation for a family  of solvable but non-nilpotent evolution algebras, which will play a key role throughout our study. 

\begin{notation}
We will denote by $\mathcal{T}_{\mathbb{K}}$ the set of all solvable but non-nilpotent evolution algebras with one-dimensional derived subalgebra over a field $\mathbb{K}$ with $\operatorname{char}\mathbb{K}\neq2$. Accordingly, we will say that an evolution algebra $\mathcal{E}$ of this type is an element of $\mathcal{T}_{\mathbb{K}}$, $\mathcal{E}\in\mathcal{T}_\mathbb{K}$. For simplicity in exposition, we slightly abuse notation by also considering a square matrix $A$ to be in $\mathcal{T}_{\mathbb{K}}$, $A\in\mathcal{T}_\mathbb{K}$, if it can serve as the structure matrix of an evolution algebra in $\mathcal{T}_\mathbb{K}$.
\end{notation}

This family of evolution algebras has already been introduced and characterised in \cite{CGOT_13}. Given a vector space $V$ over $\mathbb{K}$ with basis $B=\{e_1,\dots,e_n\}$ and scalars $\lambda_1,\dots,\lambda_n\in\mathbb{K}$, not all of them zero, and $\sum_{j=1}^{k}\lambda_j=0$ for a $k\in\{1,\dots,n\}$, we define $\mathcal{E}_k(\lambda_1,\dots,\lambda_n)$ as the evolution algebra with natural basis $B$ and product given by \[e_i^2=\lambda_i(e_1+\dots+e_k), \text{ for all } i=1,\dots,n.\] In fact, \cite[Proposition~2.5 \& Remark~2.6]{CGOT_13} state precisely that every complex evolution algebra in $\mathcal{T}_\mathbb{K}$ is isomorphic to one of the algebras described above with $k\geq2$ and scalars $\lambda_1,\dots,\lambda_k\in\mathbb{K}$, not all zero. The proof extends readily to any field. Nevertheless, we now present a slightly stronger characterisation.

\begin{proposition}\label{prop:T_parameters}
Every evolution algebra of $\mathcal{T}_{\mathbb{K}}$ is isomorphic to an evolution algebra $\mathcal{E}_k(\lambda_1,\dots,\lambda_n)$ with $k\geq2$ and $\lambda_1,\dots,\lambda_k\neq0$.
\end{proposition}
\begin{proof}
Consider an evolution algebra $\mathcal{E}\in\mathcal{T}_{\mathbb{K}}$. By \cite[Proposition~2.5 \& Remark~2.6]{CGOT_13}, $\mathcal{E}$ is isomorphic to $\mathcal{E}_k(\lambda_1,\dots,\lambda_n)$ with $k\geq2$ and $\lambda_1,\dots,\lambda_n\in\mathbb{K}$, not all of them zero. Then, just reordering the natural basis, there exists a natural number $2\leq m\leq k$ such that $\lambda_1,\dots,\lambda_m\neq0$ and $\lambda_{m+1},\dots,\lambda_k=0$. If $m=k$, we are done. Otherwise, we can consider the following natural basis transformation:
$f_1=\frac{1}{\lambda_1}(e_1+e_{m+1}+\dots+e_k)$ and $f_i=\frac{1}{\lambda_1}e_i$ for all $i=2,\dots,n$. Then, we get that $f_1^2=\frac{1}{\lambda_1}(e_1+\dots+e_k)=f_1+\dots+f_m$ and $f_i^2=\frac{\lambda_i}{\lambda_1^2}(e_1+\dots+e_k)=\frac{\lambda_i}{\lambda_1}(f_1+\dots+f_m)$ for all $i=2,\dots,n$. Consequently, it holds that $\mathcal{E}$ is isomorphic to 
\[\mathcal{E}_m\left(1,\frac{\lambda_2}{\lambda_1},\dots,\frac{\lambda_m}{\lambda_1},0,\dots,0,\frac{\lambda_{k+1}}{\lambda_1},\dots,\frac{\lambda_n}{\lambda_1}\right),\]
where $1,\frac{\lambda_2}{\lambda_1},\dots,\frac{\lambda_m}{\lambda_1}\neq0$, what yields the claim.
\end{proof}
\begin{remark}\label{rem:T_parameters}
	Because of Proposition~\ref{prop:T_parameters}, we may assume without loss of generality that every evolution algebra in $\mathcal{T}_{\mathbb{K}}$ is of the form $\mathcal{E}_k(\lambda_1,\dots,\lambda_n)$ with $k\geq2$ and $\lambda_1,\dots,\lambda_k\neq0$.
\end{remark}
\begin{remark}\label{rem:split}
	Every evolution algebra in $\mathcal{T}_\mathbb{K}$ splits over its annihilator. Given an evolution algebra $\mathcal{E}_k(\lambda_1,\dots,\lambda_n)$ with $k\geq2$ and $\lambda_1,\dots,\lambda_k\neq0$, the result follows from considering the ideal $\spa\{e_i\in B\colon\lambda_i\neq0\}$, which is complemented by $\ann(\mathcal{E})=\spa\{e_i\in B\colon\lambda_i=0\}$.
\end{remark}

Finally, we use the family $\mathcal{T}_\mathbb{K}$ to characterise all one-dimensional abelian ideals of an evolution algebra.
\begin{proposition}\label{prop:ab_ideals}
	Let $\mathcal{E}$ be an evolution $\mathbb{K}$-algebra. Every one-dimensional abelian ideal is either spanned by an element of the annihilator or is the derived subalgebra of a basic ideal isomorphic to an evolution algebra in $\mathcal{T}_\mathbb{K}$.
\end{proposition}
\begin{proof}
	Let $I$ be a one-dimensional abelian ideal of $\mathcal{E}$, spanned by an element $u=\sum_{i=1}^n\mu_ie_i$. If $u\notin\ann(\mathcal{E})$, there exists at least one index $k\in\supp(u)$ such that $e_k^2\neq0$. Since $I$ is a one-dimensional abelian ideal, it follows that $e_i^2$ is collinear with $u$ for all $i\in\supp(u)$, and that $u^2=(\sum_{i=1}^n\mu_ie_i)^2=0$.
	Now, define $J=\spa\{e_i\colon i\in\supp(u)\}$. Clearly, $J$ is a solvable but non-nilpotent basic ideal of $\mathcal{E}$ with $J^2=I$, completing the proof.
\end{proof}
\subsection{Preliminaries on Frattini theory} For completeness, we recall some
basic definitions and results relevant to our study, which are already within the framework
of evolution algebras. Given an evolution algebra $\mathcal{E}$, its \textit{Frattini subalgebra}, $F(\mathcal{E})$, is defined as the intersection of all maximal subalgebras of $\mathcal{E}$; and its \textit{Frattini ideal}, $\phi(\mathcal{E})$, as the largest ideal contained in $F(\mathcal{E})$. Moreover, $\mathcal{E}$ is said to be \textit{$\phi$-free} if $\phi(\mathcal{E})=0$. Analogously to group theory, $F(\mathcal{E})$ corresponds to the set of non-generators of $\mathcal{E}$ (see \cite[Theorem~1]{T_71_frat}). Moreover, it is known that the Frattini subalgebra of $\mathcal{E}$ is contained in the derived subalgebra $\mathcal{E}^2$ (see \cite[Lemma~1]{M_67_frat}; the proof works for general non-associative algebras). Actually, if $\mathcal{E}$ is nilpotent, the equality holds (see \cite[Theorem~6]{T_71_frat}).
We conclude this section by stating the following two lemmas which will be essential to our investigation.
\begin{lemma}{\cite[Lemma~4.1]{frat_towers}}\label{lem:frat1}
Let $\mathcal{E}$ be an (evolution) algebra. If $U$ is a subalgebra of $\mathcal{E}$, and $I$ is an ideal of $\mathcal{E}$ contained in $F(U)$, then $I$ is contained in $F(\mathcal{E})$.
\end{lemma}
\begin{lemma}{\cite[Lemma~7.2]{frat_towers}}\label{lem:frat2}
Let $\mathcal{E}$ be an (evolution) algebra. If $I$ is an abelian ideal of $\mathcal{E}$ such that $\phi(\mathcal{E})\cap I=0$, then there exists a subalgebra $U$ of $\mathcal{E}$ such that $\mathcal{E}=U\oplus I$.
\end{lemma}

\section{Defining the nilradical of an evolution algebra}\label{sec:3}

Traditionally, the nilradical of a commutative ring is the ideal consisting of all nilpotent elements. Similarly, the nilradical of a Lie algebra $\mathcal{L}$, $\nil(\mathcal{L})$, is its maximal nilpotent ideal, which exists since the sum of any two nilpotent ideals is also nilpotent. This notion also extends to Leibniz algebras (see \cite[Corollary 4]{bosko2011jacobson}). However, as shown in the next example, more than one maximal nilpotent ideal may exist in the context of evolution algebras.

\begin{example}\label{ex:1}
	Let $\mathcal{E}$ be the evolution algebra with natural basis $\{e_1,e_2,e_3,e_4\}$ and product given by $e_1^2=-e_2^2=e_3+e_4$ and $e_3^2=-e_4^2=e_1+e_2$. The subspaces $\mathcal{N}_1=\spa\{e_1,e_2,e_3+e_4\}$ and $\mathcal{N}_2=\spa\{e_3,e_4,e_1+e_2\}$ are two different maximal nilpotent ideals. However, $\mathcal{N}_1+\mathcal{N}_2=\mathcal{E}$ is not nilpotent.
\end{example}
Even though the previous example suggests otherwise, this does not rule out the existence of a unique maximal nilpotent ideal. In fact, we introduce the following notation.
\begin{remark} 
If an evolution algebra $\mathcal{E}$ has a unique maximal nilpotent ideal, then we call this ideal the nilradical of $\mathcal{E}$, and we denote it by $\nil(\mathcal{E})$.
\end{remark}
For instance, if $\mathcal{E}$ is nilpotent, then $\nil(\mathcal{E})=\mathcal{E}$ trivially. Moreover, the nilradical of an evolution algebra in $\mathcal{T}_\mathbb{K}$, say $\mathcal{E}_k(\lambda_1,\dots,\lambda_n)$ with $k\geq2$ and $\lambda_1,\dots,\lambda_k\neq0$ by Remark~\ref{rem:T_parameters}, exists and can be perfectly characterised in terms of $k$ and the scalars $\lambda_1,\dots,\lambda_k$.

\begin{lemma}\label{lem:lem1}
	Let $\mathcal{E}\in\mathcal{T}_\mathbb{K}$. Then, its derived subalgebra, $\mathcal{E}^2$, is contained in every maximal nilpotent ideal.
\end{lemma}

\begin{proof}
	Let $\mathcal{N}$ be a maximal nilpotent ideal of $\mathcal{E}$. Assume, for the sake of contradiction, that $\mathcal{E}^2\nsubseteq\mathcal{N}$. In this case, if $i\in\supp(\mathcal{N})$ then it necessarily holds that $e_i^2=0$. Consequently, $\mathcal{N}=\ann(\mathcal{E})$. Nevertheless, note that $\mathcal{E}^2+\mathcal{N}$ is also nilpotent since $(\mathcal{E}^2+\mathcal{N})^2=0$, which contradicts the maximality of $\mathcal{N}$ or the non-nilpotency of $\mathcal{E}$.	 
\end{proof}

\begin{theorem}\label{th:nilrad_dim1}
	Let $\mathcal{E}=\mathcal{E}_k(\lambda_1,\dots,\lambda_n)\in\mathcal{T}_\mathbb{K}$. Then, its nilradical, $\nil(\mathcal{E})$, exists. In fact, assuming, without loss of generality, that $k\geq2$ and $\lambda_1,\dots,\lambda_k\neq0$, it holds that
	\begin{align*}
	\nil\big(\mathcal{E}_k(\lambda_1,\dots,\lambda_n)\big)=\spa\{\lambda_2e_1-\lambda_1e_2,\lambda_3e_1-\lambda_1e_3,\dots,\lambda_ke_1-\lambda_1e_k,e_{k+1},\dots,e_n\}.
	\end{align*}
\end{theorem}
\begin{proof}
	For simplicity, we denote by $\mathcal{N}$ the subspace described just above and prove that it is the only maximal nilpotent ideal of  $\mathcal{E}_k(\lambda_1,\dots,\lambda_n)$ with $k\geq2$ and $\lambda_1,\dots,\lambda_k\neq0$. First, $\mathcal{N}$ is an ideal because $\mathcal{E}^2=\spa\{e_1+\dots+e_k\}\subset\mathcal{N}$. In fact, we have that
	\begin{align*}
		(\lambda_2e_1-\lambda_1e_2)&+(\lambda_3e_1-\lambda_1e_2)+\dots+(\lambda_ke_1-\lambda_1e_k)\\&
		=(\lambda_2+\lambda_3+\dots+\lambda_k)e_1-\lambda _1e_2-\lambda_1e_3-\dots-\lambda_1e_k\\&
		=-\lambda_1e_1-\lambda _1e_2-\lambda_1e_3-\dots-\lambda_1e_k
		=-\lambda_1(e_1+\dots+e_k);
	\end{align*}
	and, as $\lambda_1\neq0$ by hypothesis, we conclude that $e_1+\dots+e_k\in\mathcal{N}$. Secondly, $\mathcal{N}$ is maximal since it has codimension one. Thirdly, $\mathcal{N}$ is nilpotent since $\mathcal{N}^3=\mathcal{N}^2\mathcal{N}=0$. In fact, it holds that
	\begin{align}\label{eq:para_siguiente_th}
		(\lambda_ie_1-\lambda_1e_i)(e_1+\dots+e_k)=\lambda_i\lambda_1(e_1+\dots+e_k)-\lambda_1\lambda_i(e_1+\dots+e_k)=0,
	\end{align}
for any $2\leq i\leq k$ and $e_i(e_1+\dots+e_k)=0$ for any $k+1\leq i\leq n$. 

Next, for the sake of contradiction, suppose that there exists another maximal nilpotent ideal $\mathcal{M}$ and consider a nonzero element $u=\sum_{i=1}^n\mu_ie_i$ such that $u\in\mathcal{M}$ but $u\notin\mathcal{N}$. Then, we have that 
\begin{multline}\label{op}
	u+\frac{\mu_2}{\lambda_1}(\lambda_2e_1-\lambda_1e_2)+\dots+\frac{\mu_k}{\lambda_1}(\lambda_ke_1-\lambda_1e_k)-\mu_{k+1}e_{k+1}-\dots-\mu_ne_n\\=\left(\mu_1+\frac{\mu_2\lambda_2}{\lambda_1}+\dots+\frac{\mu_k\lambda_k}{\lambda_1}\right)e_1\neq0.
\end{multline}
Note that if \eqref{op} were equal to zero, there would be a contradiction with the fact that $u\notin\mathcal{N}$. Consequently, $\mu_1+\frac{\mu_2\lambda_2}{\lambda_1}+\dots+\frac{\mu_k\lambda_k}{\lambda_1}\neq0$. Moreover, by Lemma \ref{lem:lem1}, it holds that $e_1+\dots+e_k\in\mathcal{M}$. Additionally, it holds that $u(e_1+\dots+e_k)=0$. Otherwise, $u(e_1+\dots+e_k)=\mathbb{K}^*(e_1+\dots+e_k)$ and, consequently, $\mathcal{M}^{\langle n\rangle}\neq0$ for any $n\in\mathbb{N}$, which contradicts the nilpotency of $\mathcal{M}$. Finally, putting all this together, we get that 
\begin{align*}
	0=&\left(u+\frac{\mu_2}{\lambda_1}(\lambda_2e_1-\lambda_1e_2)+\dots+\frac{\mu_k}{\lambda_1}(\lambda_ke_1-\lambda_1e_k)-\mu_{k+1}e_{k+1}-\dots-\mu_ne_n\right)(e_1+\dots+e_k)\\=&\left(\mu_1+\frac{\mu_2\lambda_2}{\lambda_1}+\dots+\frac{\mu_k\lambda_k}{\lambda_1}\right)e_1(e_1+\dots+e_k)
	=\left(\mu_1+\frac{\mu_2\lambda_2}{\lambda_1}+\dots+\frac{\mu_k\lambda_k}{\lambda_1}\right)e_1^2,
\end{align*}
which is impossible, thus leading to a contradiction with the assumption that $\mathcal{M}$ is another maximal nilpotent ideal.
\end{proof}

The following property will also be instrumental throughout our study. 
\begin{corollary}\label{lem:annvsnilrad}
	Let $\mathcal{E}\in\mathcal{T}
	_{\mathbb{K}}$. Then, $\ann_{\mathcal{E}}(\mathcal{E}^2)\coloneqq\{x\in\mathcal{E}\colon x\mathcal{E}^2=0\}=\nil(\mathcal{E})$.
\end{corollary}
\begin{proof}
	As shown in \eqref{eq:para_siguiente_th}, $\mathcal{E}^2\nil(\mathcal{E})=0$. Consequently, $\nil(\mathcal{E})\subset\ann_{\mathcal{E}}(\mathcal{E}^2)$. Now, assume that $\nil(\mathcal{E})\varsubsetneq\ann_{\mathcal{E}}(\mathcal{E}^2)$. As $\nil(\mathcal{E})$ has codimension one and $\ann_{\mathcal{E}}(\mathcal{E}^2)$ is a subspace, then $\ann_{\mathcal{E}}(\mathcal{E}^2)=\mathcal{E}$, which is a contradiction with the non-nilpotency of $\mathcal{E}$.
\end{proof}

After this point, our main aim is to establish a good definition for the nilradical of an evolution algebra. 
\subsection{The basic nilradical of an evolution algebra}

Let $\mathcal{E}$ be an evolution algebra with a natural basis $B$. Following \cite[Definition 3.3]{EL_16}, we recall the ideals $\ann^i(\mathcal{E})$, $i\geq1$, where $\ann^1(\mathcal{E})\coloneqq\ann(\mathcal{E})$ and $\ann^i(\mathcal{E})$ with $i\geq2$ is defined by $\ann^i(\mathcal{E})/\ann^{i-1}(\mathcal{E})\coloneqq\ann(\mathcal{E}/\ann^{i-1}(\mathcal{E}))$. Equivalently, $\ann^i(\mathcal{E})\coloneqq\spa\{e\in B\colon e^2\in\ann^{i-1}(\mathcal{E})\}$ for all $i\geq2$. The chain of ideals:
\[0\subseteq\ann^1(\mathcal{E})\subseteq\dots\subseteq\ann^{r}(\mathcal{E})\subseteq\cdots\]
is called the \textit{upper annihilating series} of $\mathcal{E}$. Notice that, as we are only considering finite-dimensional evolution algebras, there exists an integer $r\geq1$ such that $\ann^r(\mathcal{E})=\ann^{r+1}(\mathcal{E})=\ann^{r+2}(\mathcal{E})=\cdots$, that is, the upper annihilating series stabilises for some $r\geq1$.
\begin{proposition}
Let $\mathcal{E}$ be an evolution algebra, and $r\geq1$ the number of steps until the upper annihilating series stabilises. Then, $\ann^r(\mathcal{E})$ is the largest basic nilpotent ideal of $\mathcal{E}$.
\end{proposition}
\begin{proof}
First, notice that there exists a unique maximal basic nilpotent ideal since the sum of basic nilpotent ideals is clearly nilpotent due to their strictly triangular structure matrix (see \cite[Theorem 2.7]{CLOR_14}). So, since $\ann^r(\mathcal{E})$ is clearly a basic nilpotent ideal, say $\ann^r(\mathcal{E})=\spa\{e_k,\dots,e_n\}$, it is clearly contained in the largest one.
For the sake of contradiction, assume that $I=\spa\{e_l,\dots,e_n\}$ with $l<k$ is a larger basic nilpotent ideal. As $I$ is a nilpotent evolution algebra by itself, its structure matrix can be supposed to be strictly upper triangular. This implies that $e_{k-1}^2\in\spa\{e_k,\dots,e_n\}=\ann^r(\mathcal{E})$, contradicting the definition of $r$.
\end{proof}
Because of the previous result, we introduce the following definition.
\begin{definition}\label{def:basic_nilrad}
The \textit{basic nilradical} of an evolution algebra $\mathcal{E}$, denoted by $\bnil(\mathcal{E})$, is defined as the largest nilpotent basic ideal of $\mathcal{E}$. Equivalently, $\bnil(\mathcal{E})=\ann^r(\mathcal{E})$, where $r$ is the number of steps until the upper annihilating series of $\mathcal{E}$ stabilises.
\end{definition}
However, the basic nilradical does not coincide with the nilradical in the particular case considered in Theorem \ref{th:nilrad_dim1}. In fact, if $\mathcal{E}\in\mathcal{T}_\mathbb{K}$ then $\bnil(\mathcal{E})=\ann(\mathcal{E})\subsetneq\nil(\mathcal{E})$. This makes us realise that the basic nilradical is far from being a suitable definition for the nilradical of an evolution algebra.

\subsection{The supersolvable nilradical of an evolution algebra}
Let $\mathcal{E}$ be an evolution algebra with a natural basis $B$. We introduce the sequence of ideals $\mathcal{N}^i(\mathcal{E})$, $i\geq1$. To define $\mathcal{N}^1(\mathcal{E})$, we consider the sum of all one-dimensional abelian ideals, which is actually the sum of a finite number of them, namely $I_{1,j}=\spa\{w_{1,j}\}$ for $j\in\varLambda_1$.
By Proposition~\ref{prop:ab_ideals}, we consider a partition of the set $\varLambda_1=\varGamma_1\sqcup\overline{\varGamma_1}$ such that $I_{1,j}\subset\ann(\mathcal{E})$ for any $j\in\varGamma_1$ (in fact, $\sum_{j\in\varGamma_1}I_{1,j}=\ann(\mathcal{E})$) and such that every $I_{1,j}$ with $j\in\overline{\varGamma_1}$ is the derived subalgebra of the basic ideal $\mathcal{E}_{1,j}=\spa\{e_k\colon k\in\supp(w_{1,j})\}\in\mathcal{T}_\mathbb{K}$ of $\mathcal{E}$.
Since the nilradical of each $\mathcal{E}_{1,j}$ with $j\in\overline{\varGamma_1}$ is characterised by Theorem \ref{th:nilrad_dim1}, we define the ideal  
\begin{align}\label{eq:def_n1}
\mathcal{N}^1(\mathcal{E})\coloneqq\sum_{j\in\overline{\varGamma_1}}\nil(\mathcal{E}_{1,j})+\sum_{j\in\varGamma_1}I_{1,j}=\sum_{I\subset\mathcal{E}\text{ basic ideal, }I\in\mathcal{T}_\mathbb{K}}\nil(I)+\ann(\mathcal{E}).
\end{align}

Inductively, to define $\mathcal{N}^{i}(\mathcal{E})$ with $i\geq2$, assume that $\mathcal{N}^{i-1}(\mathcal{E})$ is an ideal of $\mathcal{E}$ and consider the sum of all $(1+\dim{\mathcal{N}^{i-1}(\mathcal{E})})$-dimensional ideals which can be written as $\spa\{w\}+\mathcal{N}^{i-1}(\mathcal{E})$ with
$w^2\in\mathcal{N}^{i-1}(\mathcal{E})$ and $\supp(w)\cap\supp(\mathcal{N}^{i-1}(\mathcal{E}))=\emptyset$. Again, this sum can be supposed to be the sum of a finite number of them, namely $I_{i,j}=\spa\{w_{i,j}\}+\mathcal{N}^{i-1}(\mathcal{E})$ for $j\in\varLambda_i$, such that
\begin{align}\label{eq:cond_soporte}
w_{i,j}^2\in\mathcal{N}^{i-1}(\mathcal{E}) \quad\text{and}\quad \supp(w_{i,j})\cap\supp\big(\mathcal{N}^{i-1}(\mathcal{E})\big)=\emptyset \text{ for any } j\in\varLambda_i.
\end{align}
By construction, every  $\overline{I_{i,j}}=I_{i,j}/\mathcal{N}^{i-1}(\mathcal{E})=\spa\{\overline{w_{i,j}}\}$ for $j\in\varLambda_i$ is a one-dimensional abelian ideal of the quotient evolution algebra $\mathcal{E}/\mathcal{N}^{i-1}(\mathcal{E})$. Then, again applying Proposition \ref{prop:ab_ideals}, consider a partition $\varLambda_i=\varGamma_i\sqcup\overline{\varGamma_i}$ such that $\overline{I_{i,j}}\subset\ann\big(\mathcal{E}/\mathcal{N}^{i-1}(\mathcal{E})\big)$ for any $j\in\varGamma_i$ and such that every $\overline{I_{i,j}}$, with $j\in\overline{\varGamma_i}$, is the derived subalgebra of the basic ideal $\mathcal{E}_{i,j}=\spa\{\overline{e_k}\colon k\in\supp(w_{i,j})\}\in\mathcal{T}_\mathbb{K}$ of $\mathcal{E}/\mathcal{N}^{i-1}(\mathcal{E})$.
Since the nilradical of every $\mathcal{E}_{i,j}$ with $j\in\overline{\varGamma_1}$ is characterised, we define $\mathcal{N}^{i}(\mathcal{E})$ by 
\begin{align}\label{eq:def_ni}
\mathcal{N}^{i}(\mathcal{E})/\mathcal{N}^{i-1}(\mathcal{E})\coloneqq\sum_{j\in\overline{\varGamma_i}}\nil(\mathcal{E}_{i,j})+\sum_{j\in\varGamma_i}\overline{I_{i,j}},
\end{align}
which is clearly an ideal of  $\mathcal{E}/\mathcal{N}^{i-1}(\mathcal{E})$ as it is the sum of ideals of $\mathcal{E}/\mathcal{N}^{i-1}(\mathcal{E})$. Consequently, $\mathcal{N}^{i}(\mathcal{E})$ is an ideal of $\mathcal{E}$.

\begin{remark}	
	Notice that the sum of $\sum_{j\in\overline{\varGamma_1}}\nil(\mathcal{E}_{1,j})$ and $\sum_{j\in\varGamma_1}I_{1,j}$ in \eqref{eq:def_n1} is not necessarily direct. Consider the evolution algebra $\mathcal{E}=\mathcal{E}_3(1,-1,0)$. Its one-dimensional abelian ideals are $I_{1,1}=\spa\{e_3\}$ and $I_{1,2}=\spa\{e_1+e_2+e_3\}$. Notice that $I_{1,1}=\ann(\mathcal{E})$ and that $\nil(\mathcal{E}_{1,2})=\nil(\mathcal{E})=\spa\{e_1+e_2,e_3\}$. Nevertheless, $\nil(\mathcal{E}_{1,2})\cap I_{1,1}=\spa\{e_3\}\neq0$.
\end{remark}
\begin{remark}
	Notice that, unlike what happens in \eqref{eq:def_n1},
	$\sum_{j\in\varGamma_i}\overline{I_{i,j}}$ does not necessarily coincides with  $\ann\big(\mathcal{E}/\mathcal{N}^{i-1}(\mathcal{E})\big)$ in \eqref{eq:def_ni}.
\end{remark}

Subsequently, we prove that every term in the chain of ideals
\[0\subseteq\mathcal{N}^1(\mathcal{E})\subseteq\dots\subseteq\mathcal{N}^r(\mathcal{E})\subseteq\cdots\] is nilpotent and $\mathcal{E}$-supersolvable. 

\begin{proposition}\label{prop:nilp_super}
	$\mathcal{N}^{i}(\mathcal{E})$ is an $\mathcal{E}$-supersolvable nilpotent ideal of $\mathcal{E}$ for every $i\geq1$.
\end{proposition}
\begin{proof}
	We use induction on $i\geq1$. First, notice that every $\mathcal{E}_{1,j}$ is $\mathcal{E}$-supersolvable. Then, $\mathcal{N}^1(\mathcal{E})$ is $\mathcal{E}$-supersolvable since it is the sum of $\mathcal{E}$-supersolvable ideals.  Moreover, it holds that $\mathcal{N}^1(\mathcal{E})^{\langle2\rangle}=\spa\{w_{1,j}\colon j\in\varLambda_1\}$ and, by Lemma \ref{lem:annvsnilrad}, $\mathcal{N}^1(\mathcal{E})^{\langle3\rangle}=0$, what yields the nilpotency of $\mathcal{N}^1(\mathcal{E})$. Then, assume the assertion is true for $i$, that is, $\mathcal{N}^{i}(\mathcal{E})$ is an $\mathcal{E}$-supersolvable nilpotent ideal. Hence, $\mathcal{E}$-supersolvability follows straightforwardly from the fact that $\mathcal{N}^{i+1}(\mathcal{E})/\mathcal{N}^{i}(\mathcal{E})$ is clearly $(\mathcal{E}/\mathcal{N}^{i}(\mathcal{E}))$-supersolvable. For nilpotency, just notice that $\mathcal{N}^{i+1}(\mathcal{E})^{\langle2\rangle}\subset\spa\{w_{i+1,j}\colon j\in\varLambda_{i+1}\}+\mathcal{N}^{i}(\mathcal{E})$ and, by \eqref{eq:cond_soporte}, $\mathcal{N}^{i+1}(\mathcal{E})^{\langle k\rangle}\subset\mathcal{N}^{i}(\mathcal{E})^{\langle k-2\rangle}$ for any $k\geq3$. The result follows.
\end{proof}
Hence, we introduce the following definition.
\begin{definition}\label{def:super_nilp}
	Let $\mathcal{E}$ be an evolution algebra. The chain of $\mathcal{E}$-supersolvable nilpotent ideals defined by \eqref{eq:def_n1} and \eqref{eq:def_ni},
	\[0\subseteq\mathcal{N}^1(\mathcal{E})\subseteq\dots\subseteq\mathcal{N}^r(\mathcal{E})\subseteq\cdots,\] will be called the \textit{$\mathcal{E}$-supersolvable nilpotent series of $\mathcal{E}$}.
\end{definition}

Our main objective is now to prove that, given an evolution algebra $\mathcal{E}$, the term $\mathcal{N}^r(\mathcal{E})$ in the $\mathcal{E}$-supersolvable nilpotent series, where $r\geq1$ is the number of steps until the series stabilises, is the largest $\mathcal{E}$-supersolvable nilpotent ideal of $\mathcal{E}$.

Before proceeding, we first characterise the nilpotent ideals of an evolution algebra $\mathcal{E}$ in terms of its $\mathcal{E}$-supersolvable nilpotent series, which will be instrumental in the following steps. To fix notation, given a subspace $U$ of an evolution algebra $\mathcal{E}$, we denote by $\pi_U$ the linear projection $\pi_U\colon\mathcal{E}\longrightarrow\spa\{e_i\colon i\in\supp(U)\}$.

\begin{proposition}\label{prop:nilp_series}
	Let $\mathcal{E}$ be an evolution algebra, $I$ an ideal and $\mathcal{N}^k(\mathcal{E})$ the largest term of the $\mathcal{E}$-supersolvable nilpotent series contained in $I$. Then, $I$ is nilpotent if and only if there exists $l\in\mathbb{N}$ such that $I^{\langle l \rangle}\subset\mathcal{N}^k(\mathcal{E})$ and $\pi_{\mathcal{N}^k(\mathcal{E})}(w)\in\mathcal{N}^k(\mathcal{E})$ for any $w\in I$.
\end{proposition}
\begin{proof}
	First, we prove the sufficiency. By hypothesis, $\pi_{\mathcal{N}^k(\mathcal{E})}(w)\in\mathcal{N}^k(\mathcal{E})$ for all $w\in I$. Hence, just performing the corresponding elementary operations, $I$ can be written as $K+\mathcal{N}^k(\mathcal{E})$, where $K=\spa\{w_1,\dots,w_r\}$ is a subspace (not necessarily a subalgebra) such that $\supp(K)\cap\supp(\mathcal{N}^k(\mathcal{E}))=\emptyset$. Consequently, $K\cdot\mathcal{N}^k(\mathcal{E})=0$. Then, as there exist $l,s\in\mathbb{N}$ such that $I^{\langle l \rangle}\subset\mathcal{N}^k(\mathcal{E})$ and $\mathcal{N}^k(\mathcal{E})^{\langle s \rangle}=0$, it holds that 
	\begin{align*}
	I^{\langle l+s \rangle}&=(\cdots(I^{\langle l\rangle}\cdot I)\overset{s)}{\,\cdots})\cdot I\subseteq(\cdots(\mathcal{N}^k(\mathcal{E})\cdot I)\overset{s)}{\,\cdots})\cdot I\\
	&=(\cdots(\mathcal{N}^k(\mathcal{E})\cdot (K+\mathcal{N}^k(\mathcal{E})))\overset{s)}{\,\cdots})\cdot (K+\mathcal{N}^k(\mathcal{E}))=\mathcal{N}^k(\mathcal{E})^{\langle s \rangle}=0.
	\end{align*}
	
	Next, we prove the necessity of both conditions. On the one hand, if $I^{\langle l\rangle}\nsubseteq\mathcal{N}^k(\mathcal{E})$ for any $l\in\mathbb{N}$, then $I^{\langle l\rangle}\neq0$ for any $l\in\mathbb{N}$, which contradicts the nilpotency of $I$. On the other hand, if there exists an element $w\in I$ such that $\pi_{\mathcal{N}^k(\mathcal{E})}(w)\notin\mathcal{N}^k(\mathcal{E})$, then, by Lemma \ref{lem:annvsnilrad}, there exists an index $i\leq k$ and an element $w_{i,j}$ with $j\in\overline{\varGamma_i}$ such that $w w_{i,j}\in\mathbb{K}^*w_{i,j}+\mathcal{N}^{i-1}(\mathcal{E})$. Consequently, $(\cdots((ww_{i,j})w)\cdots)w\neq0$, which again contradicts the nilpotency of $I$.
\end{proof}

As shown in Example \ref{ex:1}, the sum of two nilpotent ideals in an evolution algebra is not necessarily nilpotent. However, we next prove that this holds when one of the summands is the largest term of the $\mathcal{E}$-supersolvable nilpotent series, $\mathcal{N}^r(\mathcal{E})$.

\begin{proposition}\label{prop:sumsnil}
	Let $\mathcal{E}$ be an evolution algebra. If $I$ is a nilpotent ideal, then $I+\mathcal{N}^r(\mathcal{E})$ is also a nilpotent ideal.
\end{proposition}
\begin{proof}
	We first show that $(I+\mathcal{N}^r(\mathcal{E}))^{\langle k\rangle}\subseteq I^{\langle k\rangle}+\mathcal{N}^r(\mathcal{E})$ for any $k\geq1$ by induction on $k$. When $k=1$, the result is trivially true. Then, suppose the assertion is true for $k$. Hence, it follows that 
	\begin{align*}
		(I+\mathcal{N}^r(\mathcal{E}))^{\langle k+1\rangle}&\subseteq(I^{\langle k\rangle}+\mathcal{N}^r(\mathcal{E}))(I+\mathcal{N}^r(\mathcal{E}))\\
		&=I^{\langle k+1\rangle}+\mathcal{N}^r(\mathcal{E})(I^{\langle k\rangle}+I+\mathcal{N}^r(\mathcal{E}))\subseteq I^{\langle k+1\rangle}+\mathcal{N}^r(\mathcal{E}),
	\end{align*}
	and the claim is established. Moreover, as $I$ is nilpotent by hypothesis, there exists a number $l\in\mathbb{N}$ such that $I^{\langle l \rangle}=0$ and, consequently, $(I+\mathcal{N}^r(\mathcal{E}))^{\langle l\rangle}\subseteq\mathcal{N}^r(\mathcal{E})$.
	
	Next, we claim that $\pi_{\mathcal{N}^r(\mathcal{E})}(w)\in\mathcal{N}^r(\mathcal{E})$ for any $w\in I+\mathcal{N}^r(\mathcal{E})$. Otherwise, by Lemma~\ref{lem:annvsnilrad}, there would exist an element $w\in I$, an index $i\leq r$ and an element $w_{i,j}$ with $j\in\overline{\varGamma_i}$ such that $ww_{i,j}\in\mathbb{K}^*w_{i,j}+\mathcal{N}^{i-1}(\mathcal{E})$; and moreover, $ww_{i,j}\in I$ since $I$ is an ideal. Consequently, $(\cdots(((ww_{i,j})w)w)\cdots)w\neq0$, which contradicts the nilpotency of $I$.
	
	As $(I+\mathcal{N}^r(\mathcal{E}))^{\langle l\rangle}\subseteq\mathcal{N}^r(\mathcal{E})$ and $\pi_{\mathcal{N}^r(\mathcal{E})}(w)\in\mathcal{N}^r(\mathcal{E})$ for any $w\in I+\mathcal{N}^r(\mathcal{E})$, the result follows from Proposition \ref{prop:nilp_series}.
\end{proof}

\begin{theorem}
	Let $\mathcal{E}$ be an evolution algebra, and $r\geq1$ the number of steps until the $\mathcal{E}$-supersolvable nilpotent series stabilises. Then, $\mathcal{N}^r(\mathcal{E})$ is the largest $\mathcal{E}$-supersolvable nilpotent ideal of $\mathcal{E}$.
\end{theorem}
\begin{proof}
	Notice that $\mathcal{N}^r(\mathcal{E})$ is an $\mathcal{E}$-supersolvable nilpotent ideal by Proposition \ref{prop:nilp_super}. First, we show that $\mathcal{N}^r(\mathcal{E})$ is a maximal $\mathcal{E}$-supersolvable nilpotent ideal of $\mathcal{E}$. Assume, for the sake of contradiction, that there exists an $\mathcal{E}$-supersolvable nilpotent ideal $I$ such that $\mathcal{N}^r(\mathcal{E})\subsetneq I$. Then, as $\mathcal{N}^r(\mathcal{E})$ and $I$ are $\mathcal{E}$-supersolvable, there clearly exists a $(1+\dim{\mathcal{N}^r(\mathcal{E})})$-dimensional  nilpotent ideal $J$ such that $\mathcal{N}^r(\mathcal{E})\subsetneq J\subseteq I$. As a consequence of Proposition~\ref{prop:nilp_series}, $J$ can be written as $\spa\{w\}+\mathcal{N}^r(\mathcal{E})$ with $w^2\in\mathcal{N}^r(\mathcal{E})$ and $\supp(w)\cap\supp(\mathcal{N}^r(\mathcal{E}))=\emptyset$, a contradiction with the fact that the series stabilises for $r$.
	
	Now, assume that another maximal $\mathcal{E}$-supersolvable nilpotent ideal exists, say $I$, different from $\mathcal{N}^r(\mathcal{E})$. By Proposition \ref{prop:sumsnil}, the sum $I + \mathcal{N}^r(\mathcal{E})$ is also an $\mathcal{E}$-supersolvable nilpotent ideal, which strictly contains $\mathcal{N}^r(\mathcal{E})$. However, this contradicts the maximality of $\mathcal{N}^r(\mathcal{E})$ previously shown, completing the proof.
\end{proof}

In view of the previous result, we introduce the following definition.

\begin{definition}
The \textit{supersolvable nilradical} of an evolution algebra $\mathcal{E}$, denoted by $\snil(\mathcal{E})$, is defined as the largest $\mathcal{E}$-supersolvable nilpotent ideal of $\mathcal{E}$. Equivalently, it holds that $\snil(\mathcal{E})=\mathcal{N}^r(\mathcal{E})$, where $r$ is the number of steps until the $\mathcal{E}$-supersolvable nilpotent series of $\mathcal{E}$ stabilises.
\end{definition}

For the reader's convenience, we now present an example of how to compute the supersolvable nilradical of an evolution algebra.

\begin{example}\label{ex:ejemplo}
	Let $\mathcal{E}$ be an evolution algebra with natural basis $\{e_1,e_2,e_3,e_4,e_5,e_6,e_7,e_8\}$ and structure matrix
		\[
		\newcommand*{\temp}{\multicolumn{1}{|r}{0}}
		\newcommand*{\temps}{\multicolumn{1}{|r}{1}}
		\newcommand*{\tempz}{\multicolumn{1}{|r}{-1}}
		\left(\begin{array}{rrrrrrrr}
		1&1&1&\temp&0&0&0&0\\
		1&1&1&\temp&0&0&0&0\\
		-2&-2&-2&\temp&0&0&0&0\\\cline{1-4}
		0&0&0&\temp&\temp&0&0&0\\\cline{1-6}
		1&-1&0&0&\temps&1&\temp&0\\
		4&0&2&0&\tempz&-1&\temp&0\\\cline{1-8}
		1&1&1&0&0&0&\temps&1\\
		0&0&0&0&0&0&\tempz&-1
		\end{array}
		\right)
		\]
	The one-dimensional abelian ideals of $\mathcal{E}$ are $I_{1,1}=\spa\{e_4\}$ and $I_{1,2}=\spa\{e_1+e_2+e_3\}$. In fact, $I_{1,1}=\ann(\mathcal{E})$ and $I_{1,2}$ is the derived subalgebra of the basic ideal $\mathcal{E}_{1,2}=\spa\{e_1,e_2,e_3\}=\mathcal{E}_3(1,1,-2)$. Consequently, following \eqref{eq:def_n1}, we have that 
	\[\mathcal{N}^1(\mathcal{E})=\nil(\mathcal{E}_{1,2})+I_{1,1}=\spa\{e_2-e_1,2e_1+e_3,e_4\}.\]
	Next, following condition \eqref{eq:cond_soporte}, we consider the ideals $I_{2,1}=\spa\{e_5+e_6\}+\mathcal{N}^1(\mathcal{E})$ and $I_{2,2}=\spa\{e_7+e_8\}+\mathcal{N}^1(\mathcal{E})$. Notice that $\overline{I_{2,1}}=\spa\{\overline{e_5}+\overline{e_6}\}$ and $\overline{I_{2,2}}=\spa\{\overline{e_7}+\overline{e_8}\}$ are one-dimensional abelian ideals of $\mathcal{E}/\mathcal{N}^1(\mathcal{E})$, which are the derived subalgebras of the basic ideals $\mathcal{E}_{2,1}=\spa\{\overline{e_5},\overline{e_6}\}\cong\mathcal{E}_2(1,-1)$ and $\mathcal{E}_{2,2}=\spa\{\overline{e_7},\overline{e_8}\}\cong\mathcal{E}_2(1,-1)$, respectively. Consequently, following \eqref{eq:def_ni}, we have that 
	\[\mathcal{N}^2(\mathcal{E})/\mathcal{N}^1(\mathcal{E})=\nil(\mathcal{E}_{2,1})+\nil(\mathcal{E}_{2,2})=\spa\{\overline{e_5}+\overline{e_6},\overline{e_7}+\overline{e_8}\},\]
	which implies that $\mathcal{N}^2(\mathcal{E})=\spa\{e_2-e_1,2e_1+e_3,e_4,e_5+e_6,e_7+e_8\}$. Moreover, as $\supp(\mathcal{N}^2(\mathcal{E}))=\{1,\dots,8\}$, the series clearly stabilises in this second step. Hence,
	\[\snil(\mathcal{E})=\mathcal{N}^2(\mathcal{E})=\spa\{e_2-e_1,2e_1+e_3,e_4,e_5+e_6,e_7+e_8\}.\]
\end{example}

Notice that the construction of the supersolvable nilradical is clearly inspired by Theorem~\ref{th:nilrad_dim1}. Consequently, unlike the basic nilradical, if $\mathcal{E}\in\mathcal{T}_\mathbb{K}$, then $\snil(\mathcal{E})=\nil(\mathcal{E})$. In what follows, we characterise the nilradical when the supersolvable nilradical is a maximal nilpotent ideal.

\begin{corollary}\label{cor:snil_max}
	Let $\mathcal{E}$ be an evolution algebra. Then, $\nil(\mathcal{E})=\snil(\mathcal{E})$ if and only if $\snil(\mathcal{E})$ is a maximal nilpotent ideal.
\end{corollary}
\begin{proof}
	The necessity is straightforward.
	For the sufficiency, observe that if $\snil(\mathcal{E})$ is a maximal nilpotent ideal, it must be unique. Indeed, if there existed another maximal nilpotent ideal $I$, then $\snil(\mathcal{E})+I$ would also be nilpotent by Proposition~\ref{prop:sumsnil}, contradicting the maximality of $\snil(\mathcal{E})$.
\end{proof}

Finally, the following example shows that evolution algebras whose nilradicals are well-defined exist, but do not coincide with their supersolvable nilradicals.

\begin{example}
	Let $\mathcal{E}$ be the evolution algebra with natural basis $\{e_1,e_2,e_3,e_4\}$ and product given by $e_1^2=-e_2^2=e_1+e_2+e_3+e_4$ and $e_3^2=-e_4^2=e_1+e_2$. It is easy to check that $\mathcal{N}=\spa\{e_1+e_2,e_3,e_4\}$ is the unique maximal nilpotent ideal of $\mathcal{E}$, and consequently, $\nil(\mathcal{E})=\mathcal{N}$. However, $\mathcal{E}$ has no one-dimensional abelian ideals, so $\snil(\mathcal{E})=0$.
\end{example}

\section{The Frattini subalgebra and ideal in evolution algebras via the supersolvable nilradical}\label{sec:4}

Regarding its counterpart in both group theory and Lie algebras, the \textit{abelian socle} of an evolution algebra $\mathcal{E}$, $\asoc(\mathcal{E})$, is defined as the sum of all minimal abelian ideals of $\mathcal{E}$. Nevertheless, during this section we will mainly work with the sum of all one-dimensional abelian ideals of $\mathcal{E}$, which will be denoted by $\asoc_1(\mathcal{E})$. In fact, as a consequence of Proposition~\ref{prop:ab_ideals}, given an evolution algebra $\mathcal{E}$, it holds that
\begin{align}\label{eq:def_asoc1}
	\asoc_1(\mathcal{E})=\sum_{I\subset\mathcal{E}\text{ basic ideal, }I\in\mathcal{T}_\mathbb{K}}I^2+\ann(\mathcal{E}).
\end{align}
Clearly, $\asoc_1(\mathcal{E})$ is $\mathcal{E}$-supersolvable, and $\snil(\mathcal{E})\cap \asoc(\mathcal{E})=\asoc_1(\mathcal{E})$. Moreover, if $\mathcal{E}$ is supersolvable (if $\mathcal{E}\in\mathcal{T}_\mathbb{K}$, for instance) then, we obtain that $\asoc_1(\mathcal{E})=\asoc(\mathcal{E})$ since every ideal in this case is $\mathcal{E}$-supersolvable.

In the next result, we characterise the Frattini subalgebra and the Frattini ideal in $\mathcal{T}_\mathbb{K}$.

\begin{theorem}\label{th:nilrad_abelian}
	Let $\mathcal{E}\in\mathcal{T}_\mathbb{K}$ with natural basis $\{e_1,\dots,e_n\}$. Then, the following assertions are equivalent:
	\begin{enumerate}[\rm (i)]
		\item the annihilator of $\mathcal{E}$ is of codimension two;
		\item $\mathcal{E}$ is isomorphic to $\mathcal{E}_2(1,-1,0,\dots,0)$;
		\item $\mathcal{E}$ splits over its abelian socle;
		\item $F(\mathcal{E})=\phi(\mathcal{E})=0$; and 
		\item $\nil(\mathcal{E})$ is abelian.
	\end{enumerate}
	Otherwise, $F(\mathcal{E})=\phi(\mathcal{E})=\nil(\mathcal{E})^2=\mathcal{E}^2$.
\end{theorem}
\begin{proof}
	
	(i)$\implies$(ii). Remark \ref{rem:T_parameters} implies that $\mathcal{E}$ is isomorphic to $\mathcal{E}_2(\lambda,-\lambda,0,\dots,0)$ for some $\lambda\in\mathbb{K}^*$. Thus, by performing the natural basis transformation $f_1=\frac{1}{\lambda}e_1$,  $f_2=\frac{1}{\lambda}e_2$ and $f_i=e_i$ for all $i=3,\dots,n$, the result follows.
	
	(ii)$\implies$(iii). By \eqref{eq:def_asoc1} and the fact that $\mathcal{E}$ is supersolvable, we have that $\spa\{e_1-e_2\}$ clearly complements 
	\begin{align*}
		\asoc\big(\mathcal{E}_2(1,-1,0,\dots,0)\big)=\asoc_1\big(\mathcal{E}_2(1,-1,0,\dots,0)\big)=\spa\{e_1+e_2,e_3,\dots,e_n\}.
	\end{align*}	
	
	(iii)$\implies$(iv). If $\mathcal{E}$ splits over its abelian socle, there exists a subalgebra $U\subset\mathcal{E}$ such that $\mathcal{E}=\asoc(\mathcal{E})\oplus U$. Since $\mathcal{E}$ is supersolvable, $\asoc(\mathcal{E})$ can be written as the direct sum of some one-dimensional abelian ideals, say $\asoc(\mathcal{E})=I_1\oplus\dots\oplus I_m$. Then, the subalgebra $M_i=(I_1\oplus\dots\oplus\widehat{I}_i\oplus\dots\oplus I_m)+U$ is a maximal subalgebra of $\mathcal{E}$  for all $i=1,\dots,m$. Since $F(\mathcal{E})\subseteq\mathcal{E}^2$ in general, we have that 
	\[\phi(\mathcal{E})\subseteq F(\mathcal{E})\subseteq\left(\cap_{i=1}^mM_i\right)\cap\mathcal{E}^2=U\cap\mathcal{E}^2=0,\]
	where the last equality follows from the fact that $\mathcal{E}^2$ is a one-dimensional abelian ideal and, consequently, is contained in $\asoc(\mathcal{E})$.
	
	(iv)$\implies$(v). By Theorem \ref{th:nilrad_dim1}, the square of the nilradical, $\nil(\mathcal{E})^2$, of an evolution algebra $\mathcal{E}\in\mathcal{T}_\mathbb{K}$ is either $0$ or $\mathcal{E}^2$. Since $\nil(\mathcal{E})^2$ is an ideal and $\nil(\mathcal{E})$ is nilpotent, Lemma~\ref{lem:frat1} and \cite[Theorem 6]{T_71_frat} imply that $\nil(\mathcal{E})^2=F(\nil(\mathcal{E}))\subset F(\mathcal{E})=0$.
	
	(v)$\implies$(i). For the sake of contradiction, consider an evolution algebra $\mathcal{E}\in\mathcal{T}_\mathbb{K}$ such that $\codim\big(\ann(\mathcal{E})\big)>2$, that is, an evolution algebra $\mathcal{E}=\mathcal{E}_k(\lambda_1,\dots,\lambda_n)$ with $k\geq3$ and $\lambda_1,\dots,\lambda_k\neq0$. Then, by Theorem \ref{th:nilrad_dim1}, we have $\lambda_2e_1-\lambda_1e_2,\lambda_3e_1-\lambda_1e_3\in\nil(\mathcal{E})$. However, $(\lambda_2e_1-\lambda_1e_2)(\lambda_3e_1-\lambda_1e_3)=\lambda_2\lambda_3e_1^2\neq0$, which leads a contradiction. 
	Furthermore, in this case, we have that 
	$\mathcal{E}^2=\nil(\mathcal{E})^2=F(\nil(\mathcal{E}))\subset F(\mathcal{E})\subset\mathcal{E}^2$.
	Since $\mathcal{E}^2$ is an ideal, then $F(\mathcal{E})=\phi(\mathcal{E})=\mathcal{E}^2$, and the final assertion follows.
\end{proof}

Next, our goal is to develop a necessary condition for an evolution algebra to be $\phi$-free, using the basic and supersolvable nilradicals. However, we should be aware of one of the main weaknesses of this supersolvable nilradical compared to Lie algebras: in general, its square is not guaranteed to be an ideal, as shown in the following example.

\begin{example}
	Let $\mathcal{E}$ be the evolution algebra with natural basis $\{e_1,e_2,e_3,e_4,e_5\}$ and product given by $e_1^2=-e_2^2=e_1+e_2+e_3$, $e_3^2=0$, $e_4^2=e_1+e_2+e_3+e_4+e_5$ and $e_5^2=-e_3-e_4-e_5$. Actually, its supersolvable nilradical is $\snil(\mathcal{E})=\spa\{e_1+e_2,e_3,e_4+e_5\}$. However, its square $\snil(\mathcal{E})^2=\spa\{e_1+e_2\}$ is not an ideal.
\end{example}

\begin{theorem}\label{th:phi-free1}
	Let $\mathcal{E}$ be an evolution algebra. If $\mathcal{E}$ is $\phi$-free, then the following hold:
	\begin{enumerate}[\rm (i)]
		\item $\bnil(\mathcal{E})=\ann(\mathcal{E})$; and
		\item if $\snil(\mathcal{E})^2$ is an ideal, then $\snil(\mathcal{E})=\asoc_1(\mathcal{E})$. 
	\end{enumerate}
\end{theorem}
\begin{proof}
	For (i), notice that $\bnil(\mathcal{E})^2$ is always an ideal. In fact, $\mathcal{E}\cdot\bnil(\mathcal{E})^2\subset\mathcal{E}\cdot\bnil(\mathcal{E})=\bnil(\mathcal{E})^2$. 
	Hence, by Lemma \ref{lem:frat1}, we have that $\phi(\bnil(\mathcal{E}))=\bnil(\mathcal{E})^2\subset\phi(\mathcal{E})=0$. Consequently,  as the annihilator is the largest basic abelian ideal, the result follows.
	
	The proof of (ii) is modelled on \cite[Theorem 7.4]{frat_towers}. As $\snil(\mathcal{E})^2$ is an ideal by hypothesis, by Lemma \ref{lem:frat1}, we have that $\phi(\snil(\mathcal{E}))=\snil(\mathcal{E})^2\subset\phi(\mathcal{E})=0$. Moreover, by Lemma~\ref{lem:frat2}, there exists a subalgebra $U\subset\mathcal{E}$ such that $\mathcal{E}=\asoc_1(\mathcal{E})\oplus U$. Then, as $\asoc_1(\mathcal{E})\subset\mathcal{N}^1(\mathcal{E})$ in general, it holds that
	\begin{align}\label{cont}
		\snil(\mathcal{E})=\mathcal{E}\cap\snil(\mathcal{E})=\asoc_1(\mathcal{E})\oplus\big(U\cap\snil(\mathcal{E})\big).
	\end{align}
	Next, we claim that $U\cap\snil(\mathcal{E})$ is an abelian ideal of $\mathcal{E}$. As $U\cap\snil(\mathcal{E})$ is an ideal of $U$, it holds that
	\begin{align*}
		\mathcal{E}\cdot\big(U\cap\snil(\mathcal{E})\big)&=\asoc_1(\mathcal{E})\cdot\big(U\cap\snil(\mathcal{E})\big)+U\cdot\big(U\cap\snil(\mathcal{E})\big)\\
		&\subset\asoc_1(\mathcal{E})\snil(\mathcal{E})+U\cap\snil(\mathcal{E})\subset\snil(\mathcal{E})^2+U\cap\snil(\mathcal{E})\\
		&=U\cap\snil(\mathcal{E}).
	\end{align*}
	Moreover, $\big(U\cap\snil(\mathcal{E})\big)^2\subset\snil(\mathcal{E})^2=0$. Therefore, $U\cap\snil(\mathcal{E})=0$ or $U\cap\snil(\mathcal{E})$ contains a one-dimensional abelian ideal of $\mathcal{E}$ (which could be itself). Otherwise, we would contradict the $\mathcal{E}$-supersolvability of $\snil(\mathcal{E})$. The latter implies that
	\[\asoc_1(\mathcal{E})\cap(U\cap\snil(\mathcal{E}))\neq0,\]
	which contradicts \eqref{cont}. Consequently, $U\cap\snil(\mathcal{E})=0$ and $\snil(\mathcal{E})=\asoc_1(\mathcal{E})$.
\end{proof}

The following result facilitates the validation of the condition $\snil(\mathcal{E}) = \asoc_1(\mathcal{E})$.

\begin{proposition}\label{prop:snil_ab}
	Let $\mathcal{E}$ be an evolution algebra. Then, $\snil(\mathcal{E})=\asoc_1(\mathcal{E})$ if and only if $\snil(\mathcal{E})^2=0$ and $\snil(\mathcal{E})=\mathcal{N}^1(\mathcal{E})$.
\end{proposition}
\begin{proof}
	First, assume that $\snil(\mathcal{E})=\asoc_1(\mathcal{E})$. Since $\asoc_1(\mathcal{E})$ is abelian and, in general, $\asoc_1(\mathcal{E})\subset\mathcal{N}^1(\mathcal{E})$, we are done.
	
	Conversely, as $\snil(\mathcal{E})=\mathcal{N}^1(\mathcal{E})$ and is abelian, the nilradical of every basic ideal $I\in\mathcal{T}_\mathbb{K}$ of $\mathcal{E}$ must be abelian. Equivalently, each $I$ must be isomorphic to $\mathcal{E}_2(1,-1,0,\dots,0)$ by Theorem~\ref{th:nilrad_abelian}. Notice that the nilradical of $\mathcal{E}_2(1,-1,0,\dots,0)$ can be written as the sum of one-dimensional abelian ideals. In fact,
	$$\nil\big(\mathcal{E}_2(1,-1,0,\dots,0)\big)=\spa\{e_1+e_2\}+\spa\{e_3\}+\dots+\spa\{e_n\}.$$
	Thus, the result follows.
\end{proof}

Notice that the converse of Theorem \ref{th:phi-free1} is not true in general.
\begin{example}
	Let $\mathcal{E}$ be the evolution algebra with natural basis $\{e_1,e_2,e_3\}$ and product $e_1^2=-e_2^2=e_1+e_2$ and $e_3^2=e_2$. In fact, $\asoc_1(\mathcal{E})=\snil(\mathcal{E})=\spa\{e_1+e_2\}$ but $\phi(\mathcal{E})=\spa\{e_1,e_2\}\neq0$.
	 
	\begin{minipage}{0.6\textwidth}
			\begin{center}
				\begin{tabular}{| l | l |}
					\hline
					\textbf{Subalg. of dim. $1$} & \textbf{Subalg. of dim. $2$}  \\
					\hline
					$\spa\{e_1+e_2\}$  & $\spa\{e_1,e_2\}$ \\
					$\spa\{e_1-e_2\}$   & \\
					\hline
				\end{tabular}
			\end{center}
		\end{minipage}
		\hspace{1cm}
		\begin{minipage}{0.3\textwidth}
			\begin{center}
				\begin{tikzpicture} [scale=0.8]
					\draw[thick] (-4,0.5) -- (-4,1.5);	
					\draw[thick] (-4,0.5) -- (-4.8,-0.5);	
					\draw[thick] (-4,0.5) -- (-3.2,-0.5);	
					\draw[thick] (-4,-1.5) -- (-4.8,-0.5);	
					\draw[thick] (-4,-1.5) -- (-3.2,-0.5);	
					\draw[fill=white,thick=black] (-4,-1.5) circle [radius=0.15];
					\draw[fill=white,thick=black](-4.8,-0.5) circle [radius=0.15];
					\draw[fill=white,thick=black](-3.2,-0.5) circle [radius=0.15];
					\draw[fill=white,thick=black](-4,0.5) circle [radius=0.15];
					\draw[fill=white,thick=black](-4,1.5) circle [radius=0.15];
				\end{tikzpicture}
			\end{center}
		\end{minipage}
\end{example}

We now focus on a specific case in which the equivalence holds even when we omit the hypothesis of the square of the supersolvable nilradical being ideal, that is, when $\supp\big(\snil(\mathcal{E})\big)=\supp(\mathcal{E})$ as in Example~\ref{ex:ejemplo}.

\begin{theorem}\label{th:equiv_frat}
Let $\mathcal{E}$ be an evolution algebra with natural basis $B=\{e_1,\dots,e_n\}$ such that $\supp\big(\snil(\mathcal{E})\big)=\supp(\mathcal{E})$. Then, the following assertions are equivalent:
\begin{enumerate}[\rm (i)]
\item $\mathcal{E}$ is $\phi$-free;
\item $\mathcal{E}$ splits over $\asoc_1(\mathcal{E})$;
\item $\mathcal{E}$ splits over its annihilator, and its complement can be seen as the direct sum of copies of $\mathcal{E}_2(1,-1)$; that is, $\mathcal{E}=K\oplus\ann(\mathcal{E})$ where $K\cong\bigoplus_{i=1}^m\mathcal{E}_2(1,-1)$ with $m\leq\lfloor\frac{n}{2}\rfloor$;
\item $\snil(\mathcal{E})^2=0$ and $\snil(\mathcal{E})=\mathcal{N}^1(\mathcal{E})$.
\end{enumerate}
\end{theorem}
\begin{proof}
(i)$\implies$(ii). This follows directly from Lemma \ref{lem:frat2}.

(ii)$\implies$(iii). Assume that $\mathcal{E}$ splits over $\asoc_1(\mathcal{E})$, meaning that there exists a subalgebra $U$ of $\mathcal{E}$ such that $\mathcal{E}=\asoc_1(\mathcal{E})\oplus U$. By Remark \ref{rem:id}, we define the basic ideal $I=\spa\{e_i\colon i\in\supp(\asoc_1(\mathcal{E}))\}$. Now, by a combination of Remarks~\ref{rem:T_parameters} and \ref{rem:split}, Proposition~\ref{prop:ab_ideals} and Theorem~\ref{th:nilrad_abelian} (iii), it follows that 
$$I\cong\left(\bigoplus_{i=1}^m\mathcal{E}_2(1,-1)\right)\oplus\ann(\mathcal{E}).$$
Thus, without loss of generality, if we set $\dim{(\ann(\mathcal{E}))}=r$, we can write that
\[\asoc_1(\mathcal{E})=\mathcal{N}^1(\mathcal{E})=\spa\{e_1+e_2,e_3+e_4,\dots,e_{2m}+e_{2m+1},e_{2m+2},\dots,e_{2m+r+1}\}.\]
Furthermore, since $\asoc_1(\mathcal{E})\subset I$, it holds that $I=\mathcal{E}\cap I=\asoc_1(\mathcal{E})\oplus(U\cap I)$, where we can write that $U\cap I=\spa\{u_1,\dots,u_m\}$ with $u_i\in\mathbb{K}^*(e_{2i-1}-e_{2i})+\asoc_1(\mathcal{E})$ in such way that $u_i^2=0$ for all $1\leq i\leq m$.

Now, we claim that $\mathcal{E}=I$. Denote $k=2m+r+2$ and assume, for the sake of contradiction, that $k\leq n$. Then, since $\mathcal{E}=\asoc_1(\mathcal{E})\oplus U$, notice that there must also exist elements $u_k=e_k+w_k,\dots,u_n=e_n+w_n\in U$, where $w_i\in\asoc_1(\mathcal{E})$ for all $k\leq i\leq n$.
Now, distinguish the following two cases, both of which will lead to a contradiction:

{\bf Case (a):} \underline{There exists at least one $w_i\notin\ann(\mathcal{E})$.}
Assume, without loss of generality, that 
\[w_k=\alpha_{1}(e_1+e_2)+\dots+\alpha_{m}(e_{2m}+e_{2m+1})+\alpha_{m+1}e_{2m+2}+\dots+\alpha_{m+r}e_{2m+r+1},\]
with $\alpha_1\neq0$. Since $u_1,u_k\in U$, it holds that $u_1u_k\in\mathbb{K}^*(e_1+e_2)\subset U$, which contradicts the fact that $e_1+e_2\in\asoc_1(\mathcal{E})$.

{\bf Case (b):} \underline{Every $w_i\in\ann(\mathcal{E})$.} 
By the construction of the supersolvable nilradical, there exists at least one element in the set $\{e_k,\dots,e_n\}$, say $e_k$, such that $e_k^2\in\mathcal{N}^2(\mathcal{E})$ and
\[\supp(e_k^2)\cap\supp(\mathcal{N}^1(\mathcal{E}))\neq\emptyset.\]
Now, consider the quotient evolution algebra
$\mathcal{E}/\mathcal{N}^1(\mathcal{E})$ and the following two subcases:

{\bf Case (b.1):} If $\overline{e_k^2}=0$, then $u^2=e_k^2+w^2=e_k^2\in\mathcal{N}^1(\mathcal{E})$, a contradiction with the fact that $U$ complements $\asoc_1(\mathcal{E})$.

{\bf Case (b.2):} If $\overline{e_k^2}\neq0$, then there exist elements $e_k,\dots,e_s$, with $k<s\leq n$, such that $\spa\{\overline{e_k},\dots,\overline{e_s}\}\in\mathcal{T}_\mathbb{K}$ is a basic ideal of $\mathcal{E}/\mathcal{N}^1(\mathcal{E})$. Assume, without loss of generality, that $\spa\{\overline{e_k},\dots,\overline{e_s}\}\cong\mathcal{E}_p(\lambda_1,\dots,\lambda_{s-k+1})$ with $2\leq p\leq s-k+1$ and $\lambda_1,\dots,\lambda_p\neq0$. Then,
\[\nil\big(\spa\{\overline{e_k},\dots,\overline{e_s}\}\big)=\spa\{\lambda_2\overline{e_k}-\lambda_1\overline{e_{k+1}},\dots,\lambda_p\overline{e_k}-\lambda_1\overline{e_{k+p-1}},\overline{e_{k+p}},\dots,\overline{e_s}\}.\]
Then, we have that,
\begin{align*}
	(u_k+\dots+u_s)(\lambda_2u_k-\lambda_1u_{k+1})&=(e_k+\dots+e_{k+p-1})(\lambda_2e_k-\lambda_1e_{k+1})\in\mathcal{N}^1(\mathcal{E});\\
	&\:\ \vdots \\
	(u_k+\dots+u_s)(\lambda_{p}u_k-\lambda_1u_{k+p-1})&=(e_k+\dots+e_{k+p-1})(\lambda_{p}e_k-\lambda_1e_{k+p-1})\in\mathcal{N}^1(\mathcal{E}).
\end{align*}
Since $U$ complements $\asoc_1(\mathcal{E})=\mathcal{N}^1(\mathcal{E})$, all the previous products must be zero. Consequently, $\rank\{e_k^2,\dots,e_{k+p-1}^2\}=1$ and $e_k^2+\dots+e_{k+p-1}^2=0$. Moreover, as every $w_i\in\ann(\mathcal{E})$, we necessarily have that $e_k^2\in\mathbb{K}^*(e_k+\dots+e_{k+p-1})+\ann(\mathcal{E})$. Thus, we have that $e_k^2e_k^2=0$ and we conclude that $e_k^2\in\asoc_1(\mathcal{E})$, contradicting the definition of $I$.

(iii)$\Longleftrightarrow$(iv). Straightforward from the construction of the supersolvable nilradical.

(iii)$\implies$(i). By \cite[Theorem 4.8]{frat_towers}, we have that 
$$\phi(\mathcal{E})=\phi\big(\mathcal{E}_2(1,-1)\big)\oplus\dots\oplus\phi\big(\mathcal{E}_2(1,-1)\big)\oplus\phi\big(\ann(\mathcal{E})\big)=0.$$Thus, $\mathcal{E}$ is $\phi$-free, completing the proof.
\end{proof}

\section{Dually atomistic evolution algebras}\label{sec:5}

An evolution algebra $\mathcal{E}$ will be called \textit{dually atomistic} if every proper subalgebra of $\mathcal{E}$ is an intersection of maximal subalgebras of $\mathcal{E}$. It is easy to see that if $\mathcal{E}$ is dually atomistic, then so is every quotient algebra of $\mathcal{E}$, and if $\mathcal{E}$ is dually atomistic, then it is $\phi$-free.

In the context of non-associative algebras, Scheiderer proved in \cite{scheiderer} that every dually atomistic Lie algebra is either abelian, almost abelian or simple over a field of characteristic zero. Nevertheless, a slightly weaker version of this result, which holds over any field, was established in~\cite{paez2023subalgebra}. Specifically, if $\mathcal{L}$ is a dually atomistic Lie algebra over an arbitrary field, then $L$ is either abelian, almost abelian or semisimple.
However, an analogous result cannot be established in the context of evolution algebras. The following example presents a dually atomistic evolution algebra that is neither abelian, almost abelian, nor semisimple (a direct sum of simple evolution algebras).
\begin{example}
	Let $\mathcal{E}$ be the evolution algebra with natural basis $\{e_1,e_2,e_3\}$ and product given by $e_1^2=e_1$, $e_2^2=e_2$ and $e_3^2=\frac{1}{4}e_1+\frac{1}{4}e_2+e_3$. It is easy to show that its subalgebras are the following and that $\mathcal{E}$ is dually atomistic.
	
	\begin{minipage}{0.6\textwidth}
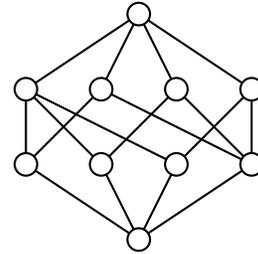
		\begin{center}
			\begin{tabular}{| l | l |}
			\hline
			\textbf{Subalg. of dim. $1$} & \textbf{Subalg. of dim. $2$}  \\
			\hline
			$\spa\{e_1\}$  & $\spa\{e_1,e_2\}$\\
			$\spa\{e_2\}$   & $\spa\{e_1,e_2+2e_3\}$  \\
			$\spa\{e_1+e_2\}$ & $\spa\{e_2,e_1+2e_3\}$  \\
			$\spa\{e_1+e_2+2e_3\}$ & $\spa\{e_3,e_1+e_2\}$ \\
			\hline
		\end{tabular}
		\end{center}
	\end{minipage}
	\hspace{1cm}
	\begin{minipage}{0.25\textwidth}
		\begin{center}
			\begin{tikzpicture}
			\draw[thick] (0,-1.5) -- (-0.5,-0.5);
			\draw[thick] (0,-1.5) -- (-1.5,-0.5);
			\draw[thick] (0,-1.5) -- (1.5,-0.5);
			\draw[thick] (0,-1.5) -- (0.5,-0.5);
			
			\draw[thick] (0,1.5) -- (-0.5,0.5);
			\draw[thick] (0,1.5) -- (-1.5,0.5);
			\draw[thick] (0,1.5) -- (1.5,0.5);
			\draw[thick] (0,1.5) -- (0.5,0.5);						
			
			\draw[thick] (-1.5,-0.5) -- (-1.5,0.5);
			\draw[thick] (-0.5,-0.5) -- (-1.5,0.5);
			\draw[thick] (0.5,-0.5) -- (-1.5,0.5);
			
			\draw[thick] (-1.5,-0.5) -- (-0.5,0.5);
			\draw[thick] (1.5,-0.5) -- (-0.5,0.5);
			
			\draw[thick] (-0.5,-0.5) -- (0.5,0.5);
			\draw[thick] (1.5,-0.5) -- (0.5,0.5);
			
			\draw[thick] (1.5,-0.5) -- (1.5,0.5);
			\draw[thick] (0.5,-0.5) -- (1.5,0.5);

			\draw[fill=white,thick=black](0.5,-0.5) circle [radius=0.15];
			\draw[fill=white,thick=black](1.5,-0.5) circle [radius=0.15];
			\draw[fill=white,thick=black](-0.5,-0.5) circle [radius=0.15];
			\draw[fill=white,thick=black](-1.5,-0.5) circle [radius=0.15];

			\draw[fill=white,thick=black](0.5,0.5) circle [radius=0.15];
			\draw[fill=white,thick=black](1.5,0.5) circle [radius=0.15];
			\draw[fill=white,thick=black](-0.5,0.5) circle [radius=0.15];
			\draw[fill=white,thick=black](-1.5,0.5) circle [radius=0.15];

			\draw[fill=white,thick=black] (0,-1.5) circle [radius=0.15];
			
			\draw[fill=white,thick=black] (0,1.5) circle [radius=0.15];
		\end{tikzpicture}
		\end{center}
	\end{minipage}
\end{example}	
 
Our goal in this section is to apply the concepts previously developed to study the dually atomistic property within specific families of evolution algebras. We exclude the semisimple case from our analysis since semisimple evolution algebras are semiprime and thus have a trivial nilradical.  Additionally, abelian evolution algebras are clearly dually atomistic. 
\begin{remark}
For example, since every dually atomistic evolution algebra $\mathcal{E}$ is $\phi$-free, Theorem~\ref{th:phi-free1} provides two straightforward necessary conditions: $\bnil(\mathcal{E})=\ann(\mathcal{E})$; and if $\snil(\mathcal{E})^2$ is an ideal, then $\snil(\mathcal{E})=\asoc_1(\mathcal{E})$. 
\end{remark}

In particular, this section aims to prove the following classification result.

\begin{theorem}\label{th:aa_dually}
Let $\mathcal{E}$ be an evolution algebra that is either almost abelian or satisfies $\supp\big(\snil(\mathcal{E})\big)=\supp(\mathcal{E})$. If $\mathcal{E}$ is dually atomistic, then it is isomorphic to one of the following pairwise non-isomorphic almost abelian evolution algebras: 
\begin{itemize}
	\item $\mathcal{E}_2(1,-1)\colon e_1^2=-e_2^2=e_1+e_2$.
	\item $\mathcal{E}_{n,1}\colon e_1^2=e_1,e_2^2=\dots=e_n^2=0$, with $n\in\mathbb{N}$.
\end{itemize}
\end{theorem}

The proof of this theorem will be a consequence of the results which follow.

Almost abelian evolution algebras had not been previously considered. So, we now provide their characterisation. To do so, we first introduce the concept of almost basic abelian evolution algebras, which will be essential for our purpose.
\begin{definition}
	An evolution algebra $\mathcal{E}$ will be called \textit{almost basic abelian} if it has an abelian basic ideal of codimension one, that is, its annihilator is of codimension one.
\end{definition}

\begin{proposition}\label{prop:char_aa}
	Let $\mathcal{E}$ be an almost abelian evolution algebra. Then, $\mathcal{E}$ is almost basic abelian, nilpotent, or lies in $\mathcal{T}_\mathbb{K}$ and its annihilator is of codimension two.
\end{proposition}
\begin{proof}
Assume that $\mathcal{E}$ is not almost basic abelian and let $I=\spa\{u_i\colon 1\leq i\leq n-1\}$ with $u_i=\sum_{j=1}^{n}\mu_{ij}e_j$, $\mu_{ij}\in\mathbb{K}$ be an abelian ideal of codimension one. Without loss of generality, suppose that the matrix $(\mu_{ij})_{i,j=1}^{n-1,n}$ is in reduced row echelon form. As $\mathcal{E}$ is not almost basic abelian then $I$ is not basic. Consequently, there exists $e_k\in B$  such that $I=\spa\{e_1+\alpha_1e_k,\dots,e_{k-1}+\alpha_{k-1}e_k,e_{k+1},\dots,e_n\}$, where at least one of $\alpha_1,\dots,\alpha_{k-1}\in\mathbb{K}$ is nonzero.

Now, we claim that there is only one nonzero scalar among $\alpha_1,\dots,\alpha_{k-1}\in\mathbb{K}$. Indeed, if at least two of them were nonzero, say $\alpha_p,\alpha_q\neq0$ with $1\leq p,q\leq k-1$, then, since $I$ is abelian, we would have $$(e_p+\alpha_pe_k)(e_q+\alpha_qe_k)=\alpha_p\alpha_qe_k^2=0,$$ which implies that $e_k\in\ann(\mathcal{E})$.
Moreover, in this case, we have that $(e_i+\alpha_ie_k)^2=e_i^2=0$ for all $i=1,\dots,k-1$, meaning that $\mathcal{E}$ is abelian, a contradiction. Hence, there must be exactly one nonzero scalar, say $\alpha_p\neq0$. Consequently, we necessarily have $e_p^2=-\alpha_p^2e_k^2\neq0$ and $e_i^2=0$ for any $i\neq p,k$. Therefore, if $e_p^2,e_k^2\in\spa\{e_i\in B\colon i\neq p,k\}$, then $\mathcal{E}$ is nilpotent but not almost basic abelian. Otherwise, if $e_p^2,e_k^2\in\mathbb{K}^*(e_p+\alpha_pe_k)+\spa\{e_i\in B\colon i\neq p,k\}$, then, via a suitable natural basis transformation, it is easy to see that $\mathcal{E}$ is isomorphic to $\mathcal{E}_2(1,-1,0\dots,0)$, what yields the claim.
\end{proof}

As a consequence of the previous result, the study of the dually atomistic property in almost abelian evolution algebras reduces to three specific cases: the nilpotent case, the class of almost basic abelian evolution algebras, and the family $\mathcal{T}_\mathbb{K}$. Before proceeding, note that any non-abelian nilpotent evolution algebra $\mathcal{E}$ has a nontrivial Frattini subalgebra, $F(\mathcal{E})=\mathcal{E}^2\neq0$, which ensures that it is not dually atomistic. Next, we study this property in the almost basic abelian case.
\begin{proposition}\label{prop:dually}
	Let $\mathcal{E}$ be an $n$-dimensional almost basic abelian evolution $\mathbb{K}$-algebra. Then, $\mathcal{E}$ is isomorphic to one of the following pairwise non-isomorphic evolution algebras:
	\begin{itemize}
		\item $\mathcal{E}_{n,1}\colon e_1^2=e_1,e_2^2=\dots=e_n^2=0$;
		\item $\mathcal{E}_{n,2}\colon e_1^2=e_2,e_2^2=\dots=e_n^2=0$.
	\end{itemize}
	Moreover, if $\mathcal{E}$ is dually atomistic, then it is necessarily isomorphic to $\mathcal{E}_{n,1}$.
\end{proposition}
\begin{proof}
	Without loss of generality, assume that the product of $\mathcal{E}$ is given by $e_1^2=\sum_{i=1}^{n}\alpha_ie_i$ and $e_2^2=\dots=e_n^2=0$. If $\alpha_1\neq0$, then we can define the element $x=\sum_{i=1}^{n}\frac{\alpha_i}{\alpha_1^2}e_i$, which is clearly idempotent. Thus, by considering the natural basis $\{x,e_2,\dots,e_n\}$, it follows that $\mathcal{E}\cong\mathcal{E}_{n,1}$. Otherwise, suppose that  $\alpha_1=0$ and consider the lowest index $k$ such that $\alpha_k\neq0$. Then, by considering the natural basis $\{e_1,e_1^2,e_2,\dots,\widehat{e_k},\dots,e_n\}$, it follows that $\mathcal{E}\cong\mathcal{E}_{n,2}$.
	Moreover, notice that $\mathcal{E}_{n,2}$ is not dually atomistic. Since it is nilpotent, we have that $F(\mathcal{E}_{n,2})=\phi(\mathcal{E}_{n,2})=\mathcal{E}_{n,2}^2=\spa\{e_2\}\neq0$.
	
	Next, we check that $\mathcal{E}_{n,1}$ is dually atomistic. Its subalgebras are of one of the following types: 
	$U=\spa\{u_1,\dots,u_m\}$ or  $V=\spa\{e_1\}+U$, where $u_1,\dots,u_m\in\spa\{e_2,\dots,e_n\}$.
	Both types can be easily expressed as the intersection of maximal subalgebras. First, consider a linear independent subset $\{u_{m+1},\dots,u_{n-1}\}$ such that $\{u_1,\dots,u_m,u_{m+1},\dots,u_{n-1}\}$ is a basis of $\spa\{e_2,\dots,e_n\}$. Then, we have that
	\begin{align*}
		V&=\bigcap_{i=m+1}^n\spa\{e_1,u_1,\dots,u_m,u_{m+1},\dots,\widehat{u_i},\dots,u_{n-1}\}\quad\text{and}\\\quad
		U&=V\cap\spa\{e_2,\dots,e_n\}.
	\end{align*}
	The result follows.
\end{proof}

Finally, we fully characterise dually atomistic evolution algebras with $\supp\big(\snil(\mathcal{E})\big)=\supp(\mathcal{E})$ (a family that includes $\mathcal{T}_\mathbb{K}$) through the following two technical lemmas.

\begin{lemma}\label{lem:max_sub}
	Consider the evolution algebra $\mathcal{E}_2(1,-1,0,\dots,0)$. Then, $M=\spa\{e_1-e_2,e_3,\dots,e_n\}$ is the only maximal subalgebra such that $e_1+e_2\notin M$. 
\end{lemma}
\begin{proof}
First, note that $\mathcal{E}_2(1,-1,0,\dots,0)$ is clearly supersolvable, then all maximal subalgebras have codimension one. Thus, consider a maximal subalgebra $M=\spa\{u_i\colon 1\leq i\leq n-1\}$ with $u_i=\sum_{j=1}^{n}\mu_{ij}e_j$, $\mu_{ij}\in\mathbb{K}$ such that $e_1+e_2\notin M$. Assume that the matrix $(\mu_{ij})_{i,j=1}^{n-1,n}$ is in reduced row echelon form. Now, we claim that $\mu_{22}=0$. Otherwise, $\mu_{22}=1$ and there would exist an element $e_k$ of the natural basis with $k>2$ such that $M=\spa\{e_1+\alpha_1e_k,\dots,e_{k-1}+\alpha_{k-1}e_k,e_{k+1},\dots,e_n\}$ with $\alpha_1,\dots,\alpha_{k-1}\in\mathbb{K}$. However, in this case, $(e_1+\alpha_1e_k )^2=e_1^2\in M$, which contradicts the fact that $e_1+e_2\notin M$. Then, $M=\spa\{e_1+\alpha e_2,e_3,\dots,e_n\}$ with $\alpha\in\mathbb{K}$. Moreover, it is easy to check that $M$ is a subalgebra if and only if $\alpha=\pm1$ but, as $e_1+e_2\notin M$, $\alpha=-1$ necessarily.
\end{proof}
\begin{lemma}\label{lem:special_case}
Let $\mathcal{E} = \mathcal{E}_2(1,-1) \oplus \mathcal{E}_2(1,-1)$ be the evolution algebra with natural basis $\{e_1,e_2,e_3,e_4\}$ and multiplication given by $e_1^2 = -e_2^2 = e_1 + e_2$, $e_3^2 = -e_4^2 = e_3 + e_4$. Then, every maximal subalgebra of $\mathcal{E}$ that contains the subalgebra $\spa\{e_1 + e_2 + e_3 + e_4\}$ also contains $\mathcal{E}^2$.
\end{lemma}
\begin{proof}
First, note that $\mathcal{E}$ is supersolvable, and so all maximal subalgebras have codimension one. Additionally, observe that $\spa\{e_1+e_2,e_3,e_4\}$ and $\spa\{e_1,e_2,e_3+e_4\}$ are maximal subalgebras that contain $\spa\{e_1 + e_2 + e_3 + e_4\}$, and both contain $\mathcal{E}^2 = \spa\{e_1 + e_2,e_3 + e_4\}$.
Now let $U = \spa\{u_1,u_2,u_3\}$ be a maximal subalgebra containing $\spa\{e_1 + e_2 + e_3 + e_4\}$ and assume $U$ is different from the two subalgebras above. Without loss of generality, we can write:
\[
u_1 = e_1 + e_2 + e_3 + e_4, \quad
u_2 = e_2 + \alpha e_3 + \beta e_4, \quad
u_3 = e_3 + \gamma e_4,
\]
for some scalars $\alpha, \beta, \gamma \in \mathbb{K}$.
Since $U$ is closed under multiplication, we must have that
\[
u_1 u_3 = e_3^2 + \gamma e_4^2 = (1 - \gamma)(e_3 + e_4) \in U,
\]
which holds if and only if $\gamma = 1$. Consequently, as $e_1 + e_2 = u_1 - u_3$,  $U$ contains both $e_1 + e_2$ and $e_3 + e_4$, so $\mathcal{E}^2 \subseteq U$, as claimed.
\end{proof}
\begin{proposition}\label{prop:dually_tk}
	Let $\mathcal{E}$ be an evolution algebra such that $\supp\big(\snil(\mathcal{E})\big)=\supp(\mathcal{E})$. Then, $\mathcal{E}$ is dually atomistic if and only if $\mathcal{E}\cong\mathcal{E}_2(1,-1)$.
\end{proposition}
\begin{proof}
	If $\mathcal{E}\cong\mathcal{E}_2(1,-1)$, its only nonzero subalgebras are $\spa\{e_1+e_2\}$ and $\spa\{e_1-e_2\}$. Then, $\mathcal{E}$ is clearly dually atomistic.
	
	Conversely, assume that $\mathcal{E}\ncong\mathcal{E}_2(1,-1)$ and distinguish the following two cases.
	
	\textbf{Case (a):} If $\mathcal{E}\in\mathcal{T}_\mathbb{K}$, then, by Proposition~\ref{prop:snil_ab}, $\mathcal{E}$ could only be dually atomistic if the codimension of the annihilator is two. In this case, by Lemma~\ref{lem:max_sub}, the element $e_1+e_2$ is contained in all maximal subalgebras except for $M=\spa\{e_1-e_2,e_3,\dots,e_n\}$. Then, consider the subalgebra $\spa\{e_1+e_2+e_3\}$.
	In fact, $e_1+e_2+e_3\notin\spa\{e_1+e_2\}$ but $\spa\{e_1+e_2+e_3\}\nsubseteq M$, so it cannot be expressed as the intersection of maximal subalgebras, and consequently $\mathcal{E}$ is not dually atomistic.
	
	\textbf{Case (b):} If $\mathcal{E}\notin\mathcal{T}_\mathbb{K}$, then, by Theorem~\ref{th:equiv_frat}, $\mathcal{E}$ could only be dually atomistic if it can be written as $K\oplus\ann(\mathcal{E})$ where $K\cong\bigoplus_{i=1}^m\mathcal{E}_2(1,-1)$ with $m\geq2$. However, in this case, $\mathcal{E}_2(1,-1) \oplus \mathcal{E}_2(1,-1)$ is a quotient algebra of $\mathcal{E}$, which is not dually atomistic by Lemma~\ref{lem:special_case}.
\end{proof}
\begin{proof}[Proof of Theorem \ref{th:aa_dually}]
	It follows from the combination of Propositions~\ref{prop:char_aa}, \ref{prop:dually} and \ref{prop:dually_tk}.
\end{proof}

\section*{Acknowledgements}
This work was partially supported by the Agencia Estatal de Investigaci\'on (Spain),
grant PID2020-115155GB-I00 (European FEDER support included, UE), and
by the Xunta de Galicia through the Competitive Reference Groups (GRC), ED431C
2023/31.
The third author was also supported by the FPU21/05685 scholarship from the Ministerio de Educaci\'on y Formaci\'on Profesional (Spain).


\begin{thebibliography}{10}
\expandafter\ifx\csname url\endcsname\relax
  \def\url#1{\texttt{#1}}\fi
\expandafter\ifx\csname urlprefix\endcsname\relax\def\urlprefix{URL }\fi

\bibitem{B_67_cohom}
D.~W. Barnes, On the cohomology of soluble {L}ie algebras, Math. Z. 101 (1967)
  343--349.
\newline\urlprefix\url{https://doi.org/10.1007/BF01109799}

\bibitem{BG_68_Lie}
D.~W. Barnes, H.~M. Gastineau-Hills, On the theory of soluble {L}ie algebras,
  Math. Z. 106 (1968) 343--354.
\newline\urlprefix\url{https://doi.org/10.1007/BF01115083}

\bibitem{frat_leib}
C.~Batten, L.~Bosko-Dunbar, A.~Hedges, J.~T. Hird, K.~Stagg, E.~Stitzinger, A
  {F}rattini theory for {L}eibniz algebras, Comm. Algebra 41~(4) (2013)
  1547--1557.
\newline\urlprefix\url{https://doi.org/10.1080/00927872.2011.643844}

\bibitem{bosko2011jacobson}
L.~Bosko, A.~Hedges, J.~T. Hird, N.~Schwartz, K.~Stagg, Jacobson's refinement
  of {E}ngel's theorem for {L}eibniz algebras, Involve 4~(3) (2011) 293--296.
\newline\urlprefix\url{https://doi.org/10.2140/involve.2011.4.293}

\bibitem{BCS_22_natural}
N.~Boudi, Y.~Cabrera~Casado, M.~Siles~Molina, Natural families in evolution
  algebras, Publ. Mat. 66~(1) (2022) 159--181.
\newline\urlprefix\url{https://doi.org/10.5565/publmat6612206}

\bibitem{tesisyolanda}
Y.~Cabrera~Casado, Evolution algebras, Ph.D. thesis, Universidad de Málaga, 2016.

\bibitem{CKS_19_basic}
Y.~Cabrera~Casado, M.~Kanuni, M.~Siles~Molina, Basic ideals in evolution
  algebras, Linear Algebra Appl. 570 (2019) 148--180.
\newline\urlprefix\url{https://doi.org/10.1016/j.laa.2019.01.010}

\bibitem{CMMT_25}
Y.~Cabrera~Casado, D.~Mart\'in~Barquero, C.~Mart\'in~Gonz\'alez, A.~Tocino, On
  simple evolution algebras of dimension two and three. {C}onstructing simple
  and semisimple evolution algebras, Linear Multilinear Algebra 73~(3) (2025)
  507--524.
\newline\urlprefix\url{https://doi.org/10.1080/03081087.2024.2352452}

\bibitem{CSV_16_semiprime}
Y.~Cabrera~Casado, M.~Siles~Molina, M.~V. Velasco, Evolution algebras of
  arbitrary dimension and their decompositions, Linear Algebra Appl. 495 (2016)
  122--162.
\newline\urlprefix\url{https://doi.org/10.1016/j.laa.2016.01.007}

\bibitem{CGOT_13}
L.~M. Camacho, J.~R. G\'omez, B.~A. Omirov, R.~M. Turdibaev, Some properties of
  evolution algebras, Bull. Korean Math. Soc. 50~(5) (2013) 1481--1494.
\newline\urlprefix\url{https://doi.org/10.4134/BKMS.2013.50.5.1481}

\bibitem{CLOR_14}
J.~M. Casas, M.~Ladra, B.~A. Omirov, U.~A. Rozikov, On evolution algebras, Algebra Colloq. 21 (2) (2014) 331--342.
  vol.~21, 2014.
\newline\urlprefix\url{https://doi.org/10.1142/S1005386714000285}

\bibitem{EL_15}
A.~Elduque, A.~Labra, Evolution algebras and graphs, J. Algebra Appl. 14~(7)
  (2015) 1550103, 10.
\newline\urlprefix\url{https://doi.org/10.1142/S0219498815501030}

\bibitem{EL_16}
A.~Elduque, A.~Labra, On nilpotent evolution algebras, Linear Algebra Appl. 505
  (2016) 11--31.
\newline\urlprefix\url{https://doi.org/10.1016/j.laa.2016.04.025}

\bibitem{F_1885_origin}
G.~Frattini, Intorno alla generazione dei gruppi di operazioni, Rom. Acc. L.
  Rend. (4) I (1885) 281--285, 455--457.

\bibitem{H_59_grouptheory}
M.~Hall, Jr., The theory of groups, Chelsea Publishing Co., New York, 1976,
  reprinting of the 1968 edition.

\bibitem{LT_85_rest}
M.~Lincoln, D.~Towers, Frattini theory for restricted {L}ie algebras, Arch.
  Math. (Basel) 45~(5) (1985) 451--457.
\newline\urlprefix\url{https://doi.org/10.1007/BF01195370}

\bibitem{M_67_frat}
E.~I. Marshall, The {F}rattini subalgebra of a {L}ie algebra, J. London Math.
  Soc. 42 (1967) 416--422.
\newline\urlprefix\url{https://doi.org/10.1112/jlms/s1-42.1.416}

\bibitem{paez2023subalgebra}
P.~P\'aez-Guill\'an, S.~Siciliano, D.~A. Towers, On the subalgebra lattice of a
  restricted {L}ie algebra, Linear Algebra Appl. 660 (2023) 47--65.
\newline\urlprefix\url{https://doi.org/10.1016/j.laa.2022.12.004}

\bibitem{scheiderer}
C.~Scheiderer, Intersections of maximal subalgebras in {L}ie algebras, J.
  Algebra 105~(1) (1987) 268--270.
\newline\urlprefix\url{https://doi.org/10.1016/0021-8693(87)90192-X}

\bibitem{Tian_08}
J.~P. Tian, Evolution algebras and their applications, vol. 1921 of Lecture
  Notes in Mathematics, Springer, Berlin, 2008.
\newline\urlprefix\url{https://doi.org/10.1007/978-3-540-74284-5}

\bibitem{TV_06}
J.~P. Tian, P.~Vojt\v{e}chovsk\'{y}, Mathematical concepts of evolution algebras
  in non-{M}endelian genetics, Quasigroups Related Systems 14~(1) (2006)
  111--122.

\bibitem{T_71_frat}
D.~A. Towers, On the generators of a nilpotent non-associative algebra, Quart.
  J. Math. Oxford Ser. (2) 22 (1971) 545--550.
\newline\urlprefix\url{https://doi.org/10.1093/qmath/22.4.545}

\bibitem{frat_towers}
D.~A. Towers, A {F}rattini theory for algebras, Proc. London Math. Soc. (3) 27
  (1973) 440--462.
\newline\urlprefix\url{https://doi.org/10.1112/plms/s3-27.3.440}

\bibitem{ZSSS_82}
K.~A. Zhevlakov, A.~M. Slinko, I.~P. Shestakov, A.~I. Shirshov, Rings that are
  nearly associative, vol. 104 of Pure and Applied Mathematics, Academic Press,
  Inc. [Harcourt Brace Jovanovich, Publishers], New York-London, 1982,
  translated from the Russian by Harry F. Smith.

\end{thebibliography}

\end{document}